\documentclass[a4paper,11pt,onecolumn]{amsart}
\usepackage{amsfonts}
\usepackage[all]{xy} 
\usepackage{amsmath}
\usepackage{amssymb}
\usepackage{amsthm}
\usepackage{graphicx}
\usepackage{epsf}
\usepackage{latexsym}
\usepackage{verbatim}
\usepackage{color}
\input epsf.tex
\usepackage{bbm}
\usepackage{mathtools}
 \usepackage{xcolor}
 \usepackage{url}
 \DeclareMathOperator{\Deck}{Deck}
 \newcommand{\Z}{\mathbb{Z}}
\newcommand{\mate}{\perp \! \! \! \perp}
\newcommand{\mix}{\mathbin{\rotatebox[origin=c]{90}{$\ltimes$}}}

\usepackage{epstopdf}
\usepackage{subcaption}

 \usepackage{tikz-cd} 

\newtheorem{theorem}{Theorem}[section] 
\newtheorem{definition}[theorem] {Definition}
\newtheorem{lemma}[theorem]{Lemma}
\newtheorem{corollary}[theorem]{Corollary}
\newtheorem{remark}[theorem]{Remark}
\newtheorem{proposition}[theorem]{Proposition}

\newtheorem*{Zieve}{Theorem}

\theoremstyle{definition}
\newtheorem{example}[theorem]{Example}

\newtheorem{question}{Question}

\newcommand{\hatC}{\hat{\mathbb{C}}}

\makeatletter %%\def\l@subsection{\@tocline{2}{0pt}{1pc}{5pc}{}}
 \def\l@subsection{\@tocline{2}{0pt}{2pc}{6pc}{}} \makeatother 

\begin{document} 

\title{Bicritical rational maps with a common iterate}
\author{Sarah Koch}
\address{Department of Mathematics, University of Michigan, Ann Arbor, MI 48109,
U.S.A.}
\email{kochsc@umich.edu}
%%%%%%%%%%%%%%%%%%%%%%%%%%%%%%%%%%%
\author{Kathryn Lindsey}
\address{Department of Mathematics, Boston College, Chestnut Hill, MA 02467, U.S.A}
\email{kathryn.lindsey@bc.edu}
%%%%%%%%%%%%%%%%%%%%%%%%%%%%%%%%%%
\author{Thomas Sharland}
\address{Department of Mathematics and Applied Mathematical Sciences, University of Rhode Island, RI 02881, U.S.A}
\email{tsharland@uri.edu}

\begin{abstract}
Let $f$ be a degree $d$ bicritical rational map with critical point set $\mathcal{C}_f$ and critical value set $\mathcal{V}_f$. Using the group $\Deck(f^k)$ of deck transformations of $f^k$, we show that if $g$ is a bicritical rational map which shares an iterate with $f$ then $\mathcal{C}_f = \mathcal{C}_g$ and $\mathcal{V}_f = \mathcal{V}_g$. Using this, we show that if two bicritical rational maps of even degree $d$ share an iterate then they share a second iterate, and both maps belong to the symmetry locus of degree $d$ bicritical rational maps.
\end{abstract}

\maketitle

\section{Introduction}

For any integers $k,d \geq 2$, the $k$-fold iteration operator, $f \mapsto f^k$, on the set $\textrm{Rat}_d$ consisting of all rational maps of degree $d$, is injective on the complement of a Zariski closed set (\cite{Ye}).  
In this work, motivated by questions about matings of polynomials, we restrict our attention to \emph{bicritical} rational maps.  Our first main theorem is the following. 

\begin{theorem}\label{mthm1} 
Let $f$ and $g$ be distinct bicritical rational maps and suppose there exists $k \in \mathbb{N}$ such that $f^k = g^k$. Then $f$ and $g$ have the same critical points and critical values.
\end{theorem}
\noindent As demonstrated in Example \ref{ex:noSharedIterate}, the converse to Theorem \ref{mthm1} does not hold.  

In the even degree case, we show that sharing \emph{any} iterate is equivalent to sharing the second iterate.  

\begin{theorem}\label{mthm2}
Let $f$ and $g$ be distinct bicritical maps which are not power maps and of even degree $d$. If there exists $k$ such that $f^k = g^k$, then $f^2 = g^2$. Furthermore, there exists an involution $\mu$ such $g = \mu \circ f = f \circ \mu$, and so $f$ and $g$ belong to the symmetry locus $\Sigma_d$.
\end{theorem}

Work of Mike Zieve \cite{ZieveQuadratics} gives a proof of Theorem \ref{mthm2} in the case $d=2$.  The proof technique is markedly different from those used in the present paper.

  \begin{Zieve}[Zieve \cite{ZieveQuadratics}]
  Let $f$ and $g$ be quadratic rational functions with a common iterate, and let $n$ be the least positive integer for which $f^n = g^n$.  If $f$ and $g$ are not power maps and $n > 1$, then $n=2$. 
\end{Zieve}

Additional work in progress by Luallen and Zieve (\cite{ZieveLuallen}) gives alternative proofs of Theorems \ref{mthm1} and \ref{mthm2}. Their work has also obtained a number results on rational functions which share a common iterate.

To prove Theorems \ref{mthm1} and \ref{mthm2}, we consider two different groups of ``symmetries'' of a rational map $f$.  First, the well-known \emph{symmetry group} or \emph{automorphism group} of a rational map $f$, $\textrm{Aut}(f)$, is the group of all M\"{o}bius transformations $\tau$ that commute with $f$.  The degree $d$ \emph{symmetry locus} is the set of degree $d$ bicritical rational maps $f$ such that $\textrm{Aut}(f)$ is nontrivial.  As shown in \cite{MilnorBicritical}, when $d$ is odd, the symmetry locus is a reducible variety, splitting into two ``halves'' with different dynamical behaviors, while when $d$ is even, the symmetry locus is irreducible.  Second, the group that we call the \emph{deck group} of a rational map $f$, $\textrm{Deck}(f)$, consists of all M\"{o}bius transformations $\tau$ such that $\tau \circ f = f$.  The groups $\textrm{Aut}(f)$, $\textrm{Deck}(f)$, as well as other groups of symmetries, are studied by Pakovich in \cite{Sym}. In particular, for a general rational map $f$, Pakovich considers the groups
\[ 
 \mathrm{Aut}_\infty(f) = \bigcup_{k=0}^\infty \mathrm{Aut}(f^k) \quad \text{and} \quad  \Deck_\infty(f) = \bigcup_{k=0}^\infty \Deck(f^k).
\]
and shows that, except for when $f$ is a power map, these groups are finite. Furthermore, he provides methods which allow an explicit description of the groups in a number of cases.

We prove the following characterization of deck groups of iterates of bicritical rational maps. 

\begin{theorem} \label{t:cyclicOrDihedral}
Let $f$ be a bicritical rational map and $k \in \mathbb{N}$.  Then $\textrm{Deck}(f^k)$ is either cyclic or dihedral. Furthermore, if the degree of $f$ is odd, then $\Deck(f^k)$ is cyclic.
\end{theorem}

If $f$ is a bicritical rational map of degree $d$, $\textrm{Deck}(f)$ contains the order-$d$ elliptic rotation around the axis in hyperbolic $3$-space whose endpoints are the critical points of $f$. Our strategy for detecting the critical points and values of $f$ from the map $f^k$ is to exploit the group structure of $\textrm{Deck}(f)$ guaranteed by Theorem \ref{t:cyclicOrDihedral} and to distinguish the critical points of $f$ from the set of all points in $\hat{\mathbb{C}}$ fixed by some nonidentity element of $\textrm{Deck}(f)$.  

It is perhaps surprising that the proof of this statement is much harder in the degree $2$ case than in the seemingly more general case for bicritical maps of degree $d \geq 3$.  We obtain the following characterization of $\Deck(f^k)$ in the case that $f$ quadratic. 

\begin{theorem}\label{mthm3}
If $f$ is a quadratic rational map, then the possibilities for $\Deck(f^k)$ (up to isomorphism) are $\mathbb{Z}_{2^n}$ for  $n \geq 1$, $V_4$ or $D_8$, the set of symmetries of a square. Furthermore, if $f$ is not a power map then $|\Deck(f^k)| \leq 8$.
\end{theorem}

The original motivation for this study was to understand and clarify the observation communicated to the authors by John Hubbard that the quadratic symmetry locus $\Sigma_2$ contains rational maps that can be viewed as variants of matings of quadratic polynomials in which the dynamics swap which ``hemisphere'' a point belongs to. While sequels will explore this topic in greater detail, we offer the following provisional definition.  

\begin{definition} \label{def:mixedmating}
Let $F$ be a rational map of degree $d$ and suppose there exist postcritically finite degree $d$ polynomials $f$ and $g$ such that
\begin{enumerate}
\item $F^2 = (f \mate g)^2$, and
\item $F \neq f \mate g$,
\end{enumerate}
where  $f \mate g$ denotes a rational map that is a geometric mating of $f$ and $g$. 
Then we say $F$ is a \emph{mixing} of $f$ and $g$ and write $F = f \mix g$.
\end{definition}

An immediate consequence of Theorem \ref{mthm1} is that a mixing $f \mix g$ of $f$ and $g$ and the corresponding geometric mating $f \mate g$ have the same critical points and critical values.  
Theorem \ref{mthm2} implies that, in the even degree case, replacing the second iterates in Definition \ref{def:mixedmating} with $k$th iterates, for $k \geq 2$, does not introduce any additional generality.  Furthermore, it implies that if $f$ and $g$ have even degree and both geometric and mixed matings of $f$ and $g$ exist, then these matings live in the symmetry locus $\Sigma_d$.   Conceptually related constructions or definitions include Timorin's work on regluings (\cite{TimorinRegluing}), twisted matings (\cite{TwistedMatings}), Meyer's antiequators \cite{MeyerUnmating}, and work in progress by Jung on quadratic anti-matings (\cite{JungAntiMatings}).

\begin{figure}[ht]
\includegraphics[width=0.6\textwidth]{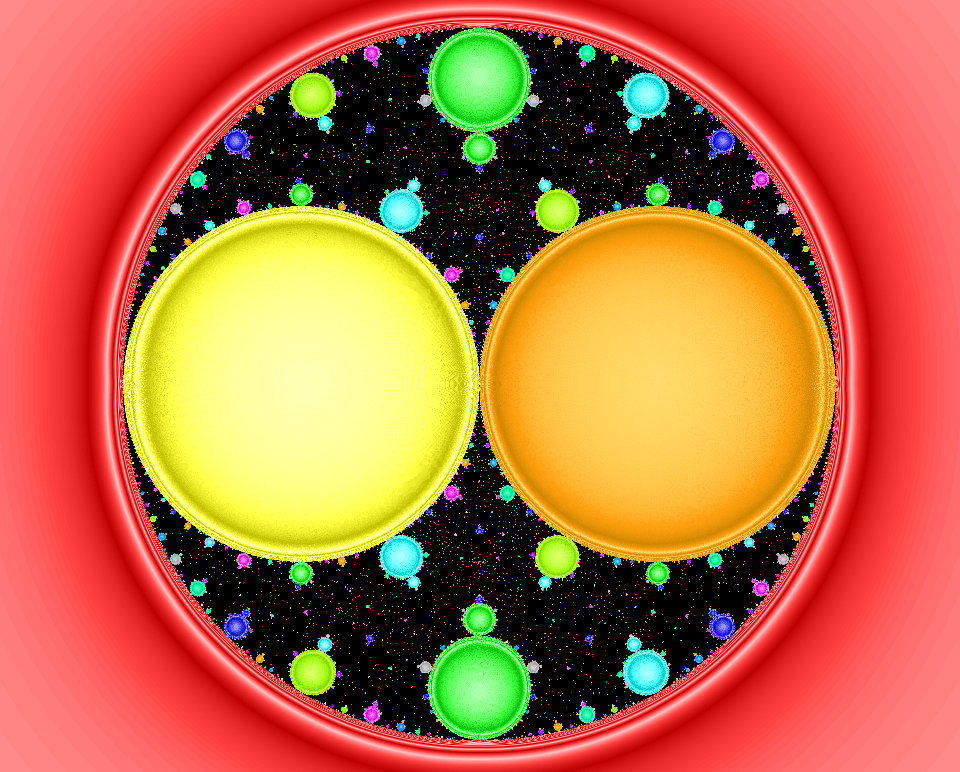}
\caption{The symmetry locus $\Sigma_2$, up to conformal conjugacy, as parameterized by $ f_c(z) = c \left( z + \frac{1}{z} \right)$ for  $c \in \mathbb{C} - \{0\}$. 
Note the space has $V_4$ symmetry. One can easily check that $f^2_c = f^2_{-c}$ and, setting $\mu(z) = -z$,  $f_{-c} = \mu \circ f_c = f_c \circ \mu$. } 
\label{f:Sym2}
\end{figure} 

The paper is organized as follows. In Section~\ref{s:deckrational}, we give some preliminary results about $\Deck(f)$ for general rational maps. In the following section, we then restrict our attention to $\Deck(f)$ where $f$ is bicritical. In Section~\ref{s:higherdegree}, we prove Theorem~\ref{mthm1} for degree $d \geq 3$. We then turn our attention to iterates of quadratic rational maps, and in Section~\ref{s:Deckiteratesquads} we undertake a deeper analysis of the possibilities for $\Deck(f^k)$ when $f$ is quadratic. This allows us, in Sections~\ref{s:DetectCritnoncritcoal} and~\ref{s:Detectcritcoal}, to prove Theorem~\ref{mthm1} for quadratics. The proof of Theorem~\ref{mthm2} is given in Section~\ref{s:sharediterates}.
Finally, in the Appendix, we revisit the space $\Sigma_2$ and present some conjectural and computational observations about matings and mixings.  

\subsection*{Acknowledgments} The authors wish to thank Xavier Buff, Eriko Hironaka, John Hubbard, Wolf Jung, Curt McMullen, Daniel Meyer and Mike Zieve for helpful conversations during the preparation of this paper. The images in the article were created with \emph{Dynamics Explorer}, \cite{DynEx}. S. Koch was partially supported by NSF grant \#2104649.  K. Lindsey was partially supported by NSF grant \#1901247.
%%%%%%%%%%%%%%%%%%%%%%%%%%%%%%%

%%%%%%%%%%%%%%%%%%%%%%%%
\section{The deck group of a rational map}\label{s:deckrational}

\begin{definition}
Let $f:\hat{\mathbb{C}} \to \hat{\mathbb{C}}$ be a rational map.  % Denote by $\textrm{M\"{o}b}$ the group of M\"{o}bius transformations. 
  The \emph{deck group} of $f$ is
\[
\Deck(f) \coloneqq \{\tau \in \textrm{Rat}_1 \mid f = f \circ \tau\}
\]
and we say that an element $\tau$ of $\Deck(f)$ is a \emph{deck transformation} of $f$.
\end{definition}

W will call elements of the group $\textrm{Aut}(f) \coloneqq \{\tau \in \textrm{Rat}_1 \mid f = \tau^{-1} \circ f \circ \tau\}$ \emph{automorphisms} of $f$. 

We will find the following notation for a fiber useful. 

\begin{definition}
For any rational map $f$ on $\hat{\mathbb{C}}$ and $z \in \hat{\mathbb{C}}$, define the \emph{fiber} of $z$ with respect to $f$ to be the set 
$$\rho_f(z) \coloneqq \{ w \in \hat{\mathbb{C}} \mid f(w) = z\}.$$
\end{definition}

The following Proposition collects some elementary facts about the general deck group of a rational map.

\begin{proposition} \label{p:elementaryfacts}
Let $f$ be a rational map of degree $d \geq 1$.  
\begin{enumerate}
\item \label{i:group} $\textrm{Deck}(f)$ is a group. 
\item \label{i:subgroup} For any $k \in \mathbb{N}$, $Deck(f)$ is a subgroup of $Deck(f^k)$. 
\item Conjugate rational maps have isomorphic deck groups.  
\item \label{i:DeckGroupFixesFibers}
Fibers are preserved by elements of the deck group.  More precisely, for any $\phi \in \textrm{Deck}(f)$ and $z \in \hat{\mathbb{C}}$, 
$$\rho_f(z) = \phi(\rho_f(z)) = \phi^{-1}(\rho_f(z)).$$
\item  \label{i:localdegreespreserved} Local degrees under $f$ are preserved by elements of the deck group. More precisely, denoting by $\textrm{deg}_f(z)$ the local degree with which a point $z$ maps forwards under $f$, we have that 
$$\textrm{deg}_f(z) = \textrm{deg}_f(\phi(z))$$ for all $\phi \in \textrm{Deck}(f)$ and $z \in \hat{\mathbb{C}}$. 
\item \label{i:orderbound} The order of $\textrm{Deck}(f)$ is at most $d$.
\item \label{i:possibilities} $\textrm{Deck}(f)$ is isomorphic to either a cyclic group, a dihedral
group, $A_4$ (the symmetry group of the tetrahedron), $S_4$ (the symmetry group of the octahedron) or $A_5$ (the symmetry group of the icosahedron).
\end{enumerate}
\end{proposition}

\begin{proof}
Conclusions \eqref{i:group}-\eqref{i:localdegreespreserved} are immediate from the definitions. The claim in \eqref{i:orderbound} follows from the uniqueness of lifts for covering spaces. Conclusion \eqref{i:possibilities} then follows from the well-known (see \cite{Klein} for a reprint of the classical reference) fact that every finite group of M\"obius transformations is isomorphic to a cyclic group, a dihedral group, $A_4$, $S_4$, or $A_5$.  
\end{proof}

We will sometimes refer to the groups $A_4$, $S_4$ and $A_5$ as the polyhedral groups. Note that in the above we consider that the Klein \emph{Vierergruppe} $V_4$ is a dihedral group. Examples of rational maps exhibiting each of the possible types of deck groups are constructed in \cite{huetal}, where the term ``half-symmetry'' is used for what we call a deck transformation. 

%%%%%%%%%%%%%%%%%%%%%%%%%%%%%%%
\section{The deck group of a bicritical rational map is cyclic or dihedral}

We will mainly be concerned with the groups $\Deck(f^k)$, where $f$ is a degree $d$ bicritical rational map. In this section we will show that the groups $\Deck(f^k)$ cannot be polyhedral groups for bicritical maps.
% This first lemma shows that $\Deck(f)$ is non-trivial for bicritical maps.

\begin{lemma} \label{l:elliptic}
Let $f$ be a bicritical rational map of degree $d \geq 2$.  Then $\textrm{Deck}(f)$ contains the elliptic M\"obius transformation that is an order $d$-rotation around the axis (geodesic in $\mathbb{H}^3$) connecting the two critical points of $f$. 
\end{lemma}

\begin{proof}
As Milnor observes in \cite{MilnorBicritical},  we can conjugate $f$ by some M\"obius transformation $\phi$ that sends the critical points of $f$ to $0$ and $\infty$; then  $\phi^{-1} \circ f \circ \phi$ has the form $$z \mapsto \frac{\alpha z^d+\beta}{\gamma z^d+\delta}$$ for some $\alpha,\beta,\gamma,\delta \in \mathbb{C}$.  Any map of this form is invariant under composition with the elliptic rotation $R(z) = ze^{2\pi i / d}$, i.e. 
\[
\phi^{-1} \circ f \circ \phi(z) = \phi^{-1} \circ f \circ \phi(R_n( z)).
\]
\end{proof}

\begin{corollary} \label{c:degreeatleast6}
Let $f$ be a bicritical rational map of degree $d \geq 6$.  Then for any $k \in \mathbb{N}$, $\textrm{Deck}(f^k)$ is either cyclic or dihedral.
\end{corollary}

\begin{proof} 
By Lemma \ref{l:elliptic}, $\textrm{Deck}(f)$ contains an element of order $d$.  Since $A_4$, $S_4$ and $A_5$ do not have any element of order $\geq 6$, the claim follows from Proposition \ref{p:elementaryfacts} part \eqref{i:possibilities}. 
\end{proof}

We now turn our attention to the case where the degree is less than or equal to $5$.

\begin{lemma} \label{l:notA_4A_5}
 Let $f$ be a bicritical rational map.  Then neither $\textrm{Aut}(f^k)$ nor $\textrm{Deck}(f^k)$ is  isomorphic to $A_4$ or $A_5$. 
\end{lemma}

\begin{proof}
Let $G$ be either of the finite groups $\textrm{Aut}(f)$ or  $\textrm{Deck}(f)$, and consider $\tau \in G$.    Then $\tau$ fixes (set-wise) the set $C \coloneqq \{c_1,c_2\}$ of critical points of $f$.  Hence, every element of $G$ fixes the set $\{c_1,c_2\}$.  Hence, there are two possibilities for the orbit $G(c_1)$ of $c_1$: either $G(c_1) = \{c_1\}$ or $G(c_1) = \{c_1,c_2\}$. 

Case 1: $G(c_1) = \{c_1\}$. Since M\"obius transformations are injective and $G(c_2) \subset \{c_1,c_2\}$ we must have $G(c_2) = \{c_2\}$.  Hence every element of $G$ fixes both $c_1$ and $c_2$ as points.  The only M\"obius transformations that have finite order are elliptic, and $G$ is a finite group.   Thus $G$ is a finite group of elliptic rotations around the axis with endpoints $c_1$ and $c_2$, i.e. $G$ is cyclic.  

  Case 2: $G(c_1) = \{c_1,c_2\}$.  Then from the Orbit-Stabilizer Theorem, we have  
$$2 = |G(c_1) |= \frac{|G|}{|\textrm{Stab}_G(c_1)|}.$$
Neither $A_4$ nor $A_5$ has a subgroup of index $2$ ($|A_4| = 12$ and $A_4$ has no subgroups of order $6$; $|A_5| = 60$ and $|A_5|$ has no subgroups of order $30$).  
\end{proof}

\begin{remark}
 Case 2 of the Proof of Lemma \ref{l:notA_4A_5} does not work for $S_4$ because $S_4$ has an index $2$ subgroup (namely, $A_4$). 
\end{remark}
 
\begin{remark} For any finite subgroup $\Gamma$ of $\textrm{Rat}_1$, Doyle and McMullen (\cite{SolvingQuintic}, Section 5) gave a recipe for constructing rational maps with $\Gamma \subseteq \mathrm{Aut}(f)$.
\end{remark}

\begin{definition} \label{d:degreepartition}
Let $f$ be a bicritical rational map of degree $d$ and let $k$ be a natural number.  The \emph{degree partition} for a point $z \in \hat{\mathbb{C}}$ with respect to $f^k$ is the ordered list of integers $\{a_{i,f^k}(z)\}_{i=0}^{k}$ where $a_{i,f^k}(z)$ is the number of points in the fiber $\rho_{f^k}(z)$ that map forward under $f$ with local degree $d^i$. 
\end{definition}

The following lemma is immediate from the definitions. 

\begin{lemma} \label{l:permutation}
Let $f$ be a bicritical rational map,  $z \in \mathbb{C}$, $k \in \mathbb{N}$, and  $\tau$ of $\textrm{Deck}(f^k)$.  
Then for each nonzero element $a_{i,f^k}(z)$ of the degree partition,  $\tau$ acts as a permutation on the set of  points in the fiber $\rho_{f^k}(z)$ that map forward under $f^k$ with local degree $d^i$.
\end{lemma}

\begin{lemma} \label{l:notamultiple}
Let $f$ be a bicritical rational map of degree $d$, let $p$ be a prime number that does not divide $d$, and let $k$ be any natural number.  Suppose there exists some element $\tau$ of $\textrm{Deck}(f^k)$ that has order $p$. Then for any point $z \in \hat{\mathbb{C}}$, there exists some element $a_{i,f^k}(z)$ of the degree partition that is not a multiple of $p$. 
\end{lemma}

\begin{proof}
Because both critical points of a bicritical, degree $d$ rational map have local degree $d$, and the total degree of $f^k$ is $d^k$, we immediately have that for any point $z \in \hat{\mathbb{C}}$,
$$d^k = \sum_{i=0}^k a_{i,f^k}(z) \cdot d^i.$$
If every $a_{i,f^k}(z)$ was a multiple of $p$, then the equation above would imply that $d^k$ is a multiple of $p$, which is a contradiction. 
\end{proof}

The following is the key observation.

\begin{proposition} \label{p:primeorders}
Let $f$ be a bicritical rational map of degree $d$, and let $p$ be a prime number that does not divide $d$.  Then for all natural numbers $k$, the deck group $\textrm{Deck}(f^k)$ has no element of order $p$.
\end{proposition}

\begin{proof}
Suppose, for a contradiction, that some element $\tau$ of $\textrm{Deck}(f^k)$ has order $p$.   Consider any point $z \in \hat{\mathbb{C}}$.  By Lemma \ref{l:notamultiple}, there exists some $j$ such that $a_{j,f^k}(z)$ is not a multiple of $p$.  By Lemma \ref{l:permutation}, $\tau$ acts as a permutation on the set, call it $S$, of the $a_{j,f^k}(z)$ many points in the fiber $\rho_{r^k}(z)$ whose local degree under $f^k$ is $a_{j,f^k}(z)$.  Since the group generated by $\tau$, $\langle \tau \rangle$ is a cyclic group of prime order $p$, the Orbit Stabilizer Theorem gives that the cardinality of the orbit under $\langle \tau \rangle$ of any point in $\mathbb{C}$  equals either $1$ or $p$.  Since $|S| = a_{j,f^k}(z)$ is not divisible by $p$, it follows that $S$ contains at least one point that is fixed by $\tau$.  
Since the point $z$ was arbitrary, this shows that every fiber contains at least one fixed point of $\tau$.  But since a non-identity M\"obius transformation can have at most 3 fixed points, this is a contradiction.  
\end{proof}

\begin{corollary} \label{c:degreeatmost5}
Let $f$ be a bicritical rational map of degree $d \leq 5$ and let $k \in \mathbb{N}$.  Then $\textrm{Deck}(f^k)$ is either cyclic or dihedral. 
\end{corollary} 

\begin{proof}
By Proposition \ref{p:elementaryfacts} part \eqref{i:possibilities}, $\textrm{Deck}(f^k)$ is either cyclic, dihedral, $A_4$, $S_4$ or $A_5$.  Lemma \ref{l:notA_4A_5} rules out $A_4$ and $A_5$.  $S_4$ has elements of order $2$ and elements of order $3$, and at least one of $2$ and $3$ does not divide $d$ for each of choice of $d$ in $\{2,\ldots,5\}$.   
Thus, Proposition \ref{p:primeorders} implies $\textrm{Deck}(f^k) \neq S_4$.  
\end{proof}

We can now prove Theorem~\ref{t:cyclicOrDihedral}.

\begin{proof}[Proof of Theorem~\ref{t:cyclicOrDihedral}]
Corollary \ref{c:degreeatmost5} gives the result in the case that degree of $f$ is  $\leq 5$, and Corollary \ref{c:degreeatleast6} gives the result for degree $\geq 6$.  If the degree of $f$ is odd, then by Proposition~\ref{p:primeorders}, $\Deck(f^k)$ cannot contain any elements of order $2$ and so cannot be dihedral. 
\end{proof}

%%%%%%%%%%%%%%%%%%%%%%%%%%%%%%%%%%%%%
\section{Detecting critical points and values of bicritical maps of degree $d \geq 3$ from their iterates}\label{s:higherdegree}

We begin with a definition.

\begin{definition}
 Let $f$ be a rational map with critical point set $\mathcal{C}_f$ (respectively critical value set $\mathcal{V}_f$). We will say that we can \emph{detect the set $\mathcal{C}_f$ (respectively $\mathcal{V}_f$) from $f^k$} if whenever $f^k = g^k$ for some bicritical rational map $g$ we have $\mathcal{C}_f = \mathcal{C}_g$ (respectively $\mathcal{V}_F = \mathcal{V}_G$).
\end{definition}

The idea behind this definition is that knowledge of $f^k$
% (and possibly, the iterates of $f^k$, if required) 
 provides enough information for us to be able to recover the sets $\mathcal{C}_f$ and $\mathcal{V}_f$. We will show that if $f$ is a bicritical rational map and $k \geq 1$, then we can always detect $\mathcal{C}_f$ and $\mathcal{V}_f$ from $f^k$. In this section we prove the following.

\begin{theorem} \label{p:detectcpscvsnonquadratic}
Fix a rational map $F$.  If there exists a bicritical rational map $f$ of degree $\geq 3$ and $k \in \mathbb{N}$ such that $f^k = F$, then we can detect the sets $\mathcal{C}_f$ and $\mathcal{V}_f$ from $F$.  Specifically, $\mathcal{C}(f)$ is the set of fixed points of any element of $\textrm{Deck}(F)$ of order at least $3$, and $\mathcal{V}_f = \{x \in \mathbb{C} \mid F^{-1}(x) \subseteq \mathcal{C}_F\}$. 
\end{theorem}

First we need a simple observation about finite cyclic groups of M\"obius transformations.

\begin{lemma} \label{l:cyclic}
Let $G$ be a finite cyclic group of M\"obius transformations.  Then there exist two distinct points $x_1$ and $x_2$ in $\hat{\mathbb{C}}$ such that every nonidentity element of $g$ is an elliptic rotation around the axis connecting $x_1$ and $x_2$. 
\end{lemma}

\begin{proof}
Two well-known facts are that i) the only M\"obius transformations of finite order are elliptic, and ii) two nonidentity M\"obius transformations commute if and only if they have the same set of fixed points or are commuting involutions each interchanging the fixed points of the other.  If $|G|=2$, we are done.  So suppose $|G| \geq 3$ and let $g \in G$ be an element of order $\geq 3$; then since $g$ commutes with every element of $G$, it must have the same set of fixed points as every nonidentity element of $G$.  
\end{proof}

The proof of Theorem~\ref{p:detectcpscvsnonquadratic} now follows from the following two lemmas.

\begin{lemma} \label{l:nonquadraticcandetectcps}
Fix a rational map $F$.  Suppose there exists at least one bicritical map $f$ of degree $d \geq 3$ and integer $k \in \mathbb{N}$ such that $f^k = F$.  Then all elements of $\textrm{Deck}(F)$ of order $\geq 3$ have the same set $C$ of fixed points and  $\mathcal{C}_f = C$. 
\end{lemma}

\begin{proof}
By Theorem \ref{t:cyclicOrDihedral}, $\textrm{Deck}(F)$ is cyclic or dihedral.  By Lemma \ref{l:elliptic}, the order $d$ elliptic rotation around the axis connecting the two critical points of $f$ is an element of $\textrm{Deck}(f)$.  Hence it is also an element of $\textrm{Deck}(F)$ by Proposition \ref{p:elementaryfacts} part \eqref{i:subgroup}.  But all elements of $F$ that have order $\geq 3$ are elliptic rotations that share the same set of fixed points by Lemma \ref{l:cyclic}  Hence $\mathcal{C}(f)$ is the set of two fixed points of any element of $\textrm{Deck}(F)$ of order at least $3$. 
\end{proof}

\begin{lemma} \label{l:nonquadraticcandetectcvs}
Let $f$ be a bicritical rational map of degree $d \geq 3$ and fix any integer $k \in \mathbb{N}$.  Then $x \in \mathcal{V}_f$ if and only if  $f^{-k}(x) \subseteq \mathcal{C}_{f^{k}}$.
\end{lemma}

\begin{proof}
If $x \in \mathcal{V}_f$ then  we have 
$f^{-k}(x) \in f^{-(k-1)}(\mathcal{C}_f)  \subseteq  \mathcal{C}_{f^{-k}}.$

Conversely suppose $x \notin \mathcal{V}_f$. We will inductively construct a sequence $ x_0, x_1, \dotsc, x_k$  with $x = x_0$ and such that $x_{i-1} = f(x_i)$ for each $i$ and no $x_i$ is a critical value of $f$. Since no $x_i$ is a critical value of $f$, it follows that no $x_i$ can be a critical point of $f$.

Since the degree of $f$ is $d \geq 3$ and $x_0$ is not a critical value, there exists $x_1 \in f^{-1}(x_0)$ such that $x_1 \notin \mathcal{V}_f$. Inductively, suppose that $x_i$ ($1 \leq i \leq k-1$) is not a critical value. Then there exists $x_{i+1} \in f^{-1}(x_i)$ such that $x_{i+1}  \notin \mathcal{V}_f$. Now consider $x_k$. It is clear that $x_k \in f^{-k}(x)$. Furthermore, we claim $x_k \notin \mathcal{C}_{f^{k}}$. The local degree of $f^{k}$ at $x_k$ is equal to the products of the local degrees of $F$ at $x_i$. Since by construction the local degree of $f$ at each $x_i$ is equal to $1$, we see that the local degree of $f^{k}$ at $x_k$ is $1$. Thus $x_k$ is not a critical point of $f^{k}$.
\end{proof}

\begin{proof}[Proof of Theorem~\ref{p:detectcpscvsnonquadratic}]
 The claim for $\mathcal{C}_f$ is contained in Lemma~\ref{l:nonquadraticcandetectcps} and the claim for $\mathcal{V}_f$ is Lemma~\ref{l:nonquadraticcandetectcvs}.
\end{proof}

%%%%%%%%%%%%%%%%%
\section{Deck groups of iterates of quadratic rational maps}\label{s:Deckiteratesquads}

It is perhaps surprising that detecting $\mathcal{C}_f$ and $\mathcal{V}_f$ in the degree $2$ case is more difficult than the general higher degree bicritical case. One reason is that the conclusion of Lemma~\ref{l:nonquadraticcandetectcvs} is not true in general for quadratic rational maps, due to what we call \emph{critically coalescing} maps. 

\begin{definition}
We will say a quadratic rational map $f$ is critically coalescing if the two critical values of $f$ share a common image. In other words, denoting the critical values of $f$ by $v_1$ and $v_2$, we have $f(v_1)=f(v_2)$.  
\end{definition}

We will need the following observation, which will be refined later on.
% The proof is similar to that of Lemma~\ref{l:nonquadraticcandetectcvs}.

\begin{lemma} \label{l:coalescing3pts}
Let $f$ be a critically coalescing quadratic rational map. Then for all $k \geq 2$, $x \in \mathcal{V}_f \cup \{ f(v_1) = f(v_2) \}$ if and only if $f^{-k}(x) \subseteq  f^{-k}(\mathcal{C}_{f^{k}})$.
\end{lemma}

\begin{proof}
Let $\beta = f(v_1) = f(v_2)$. It is simple to see that if $x \in \mathcal{V}_f \cup \{ \beta \}$ then $f^{-k}(x) \subseteq  f^{-k}(\mathcal{C}_{f^{k}})$. So suppose $x \notin \mathcal{V}_f \cup \{ \beta \}$. We will construct a sequence $x_0,\dotsc,x_k$ with $x = x_0$ and $f(x_{j}) = x_{j-1}$.  So set $x = x_0$. In particular, since $x \neq \beta$, then neither element of $f^{-1}(x_0)$ is a critical value of $f$. Furthermore, at most one preimage can be equal to $\beta$. Thus there exists $x_1 \in f^{-1}(x_0)$ such that $x_1 \notin \mathcal{V}_f \cup \{ \beta \}$. Now we can inductively find $x_2,x_3,\dotsc,x_k$ such that $x_j \notin \mathcal{V}_f \cup \{ \beta \}$ for all $j \leq k$ by the same reasoning.  As with the proof of Lemma~\ref{l:nonquadraticcandetectcvs}, we can conclude that $x_k \in f^{-k}(x)$ but $x_k \notin \mathcal{C}_{f^{k}}$.
\end{proof}

In the case where $f$ is a bicritical map of degree $d \geq 3$, we were able to detect the sets $\mathcal{C}_f$ and $\mathcal{V}_f$ by exploiting the facts that $\Deck(f^k)$ contains elements of order $d \geq 3$ and all such elements necessarily fixed the critical points of $f$ pointwise.  The case $d=2$ and $\textrm{Deck}(f^k) \cong V_4$ is harder because no deck group elements of order at least $3$ exist, and elements of order $d=2$ do not necessarily fix the critical points pointwise. The aim of the next three sections is to prove the following theorem, an analog to Lemma~\ref{l:nonquadraticcandetectcps}.

\begin{theorem} \label{t:DetectQuadraticCritPts}
Let $f$ be a quadratic rational map.  Then we can detect the critical points of $f$ from $f^k$ ($k>1)$.  Specifically:
\begin{enumerate}
\item If $f$ is not critically coalescing, then $\Deck(f^k)$ is cyclic. In particular, either
\begin{enumerate}
\item $f$ is a power map and $\Deck(f^k)$ is isomorphic to $\mathbb{Z}_{2^k}$ so that the critical points of $f$ are the fixed points of any element $\mu$ which generates $\Deck(f^k)$. 
\item $\Deck(f^k) \cong \mathbb{Z}_2$, and the critical points of $f$ are the fixed points of the unique non-identity element of $\Deck(f^k)$.
\end{enumerate}
%\item The map $f$ is a power map if and only if $| V_{f^k}| = 2$ for all $k$. Moreover $\mathcal{C}_f = \mathcal{C}_{f^k} = \mathcal{V}_{f^k} = \mathcal{V}_f$.
%\item $\Deck
%\item $\Deck(f^k) = \Deck(f) \cong C_2$, and the critical points of $f$ are the fixed points of the unique non-identity element of $\Deck(f^k)$.
\item If $f$ is critically coalescing and not conjugate to $z \mapsto \frac{z^2-1}{z^2 +1}$ then $\Deck(f^k) \cong V_4$ for all $k\geq 2$. Furthermore:
\begin{enumerate}
\item If the forward orbit of the critical values does not contain a fixed point, then the image under $f^k$ of the critical points of $f$ is distinct from the image under $f^k$ of the elements of the other special pairs of $f$.
\item If the forward orbit of the critical values does contain a fixed point $\alpha$, then the critical points of $f$ are the fixed points of $\mu$, the unique element of $\Deck(f^k)$ for which $\mu(\alpha) = \beta$, where $\beta$ is the unique element of $f^{-1}(\alpha)$ distinct from $\alpha$. 
\end{enumerate}
\item If $f$ is conjugate to $z \mapsto \frac{z^2-1}{z^2 +1}$, then $\Deck(f^2) \cong V_4$ but $\Deck(f^k) \cong D_8$ for all $k \geq 3$ and, as in case (i), the critical points of $f$ are the fixed points of any element of order $4$ in $\Deck(f^k)$.
\end{enumerate}
\end{theorem}

We will prove Theorem~\ref{t:DetectQuadraticCritPts}  at the end of Section~\ref{s:Detectcritcoal}. Here, we prove that Theorem~\ref{t:DetectQuadraticCritPts} implies Theorem~\ref{mthm3}.

 %The proof of Theorem~\ref{mthm3} will then be completed when we complete the proof of Theorem~\ref{t:DetectQuadraticCritPts} at the end of Section~\ref{s:Detectcritcoal}.

\begin{proof}[Proof of Theorem~\ref{mthm3}]
(Assuming Theorem~\ref{t:DetectQuadraticCritPts}) If $f$ is not critically coalescing then, by Theorem~\ref{t:DetectQuadraticCritPts}, $\Deck(f^k)$ is cyclic of order $2^n$ for some $n$. Moreover, $\Deck(f^k) \cong \mathbb{Z}_2$ for all $k\geq 1$ if and only if $f$ is neither a power map nor critically coalescing. If $f$ is critically coalescing, then Theorem~\ref{t:DetectQuadraticCritPts} asserts that $\Deck(f^k) \cong V_4$ for all $k \geq 2$, unless $f$ is conjugate to $z \mapsto \frac{z^2-1}{z^2 + 1}$, in which case $\Deck(f^k) \cong D_8$ for all $k \geq 3$.
\end{proof}

%\begin{theorem} \label{t:DetectQuadraticCritPts}
%Let $f$ be a quadratic rational map.  Then $\Deck(f^k)$ is cyclic or isomorphic to $V_4$ of $D_8$.  Specifically,
%\begin{enumerate}
%\item If $f$ is a power map then $\Deck(f^k) \cong \mathbb{Z}_{2^k}$.
%\item If $f$ is not critically coalescing nor a power map, then $\Deck(f^k) = \Deck(f) \cong \mathbb{Z}_2$.
%\item If $f$ is critically coalescing and not conjugate to $z \mapsto \frac{z^2-1}{z^2 +1}$ then $\Deck(f^2) \cong V_4$. \item If $f$ is conjugate to $z \mapsto \frac{z^2-1}{z^2 +1}$, then $\Deck(f^2) \cong V_4$ and $\Deck(f^k) \cong D_8$ for all $k \geq 3$.\end{enumerate}
%\end{theorem}

The proof of Theorem~\ref{t:DetectQuadraticCritPts} requires studying the groups of deck transformations for iterates of quadratic rational maps; this is the goal of the present section. In the next two sections we will use the obtained results to detect the critical points of quadratic rational maps, thus enabling us to prove the theorem. Proposition~\ref{p:characterizecritcoal} will show that $\textrm{Deck}(f^k) \cong V_4$ precisely in the critically coalescing case. 
%We have the following simple observation.
    
  \begin{lemma}  \label{l:specialpairs}
  Let $G \cong V_4$ be a group of M\"obius transformations acting on $\hat{\mathbb{C}}$. 
   Then there are precisely $6$ points in $\hat{\mathbb{C}}$ that are fixed pointwise by some non-identity element of $G$; each  of the three non-identity elements of $G$ fixes a pair of these points.
  \end{lemma}
  
  \begin{proof}
  It is easy to construct a group $G \cong V_4$ of M\"obius transformations that satisfies the conclusion.  Then the fact that every $G \cong V_4$ satisfies the conclusion follows from the well-known fact (see e.g. \cite{Beardon:Groups}) that finite groups of M\"obius transformations are isomorphic if and only if they are conjugate. 
  \end{proof}
  
  These pairs will play an important role in our strategy for detecting $\mathcal{C}_f$ and $\mathcal{V}_f$.
  
  \begin{definition} 
For $G \cong V_4$ a group of M\"obius transformations acting on $\hat{\mathbb{C}}$, the \emph{special pairs} are the three pairs of points in $\hat{\mathbb{C}}$ defined by Lemma~\ref{l:specialpairs}
  \end{definition}
 
 We will give a characterization of the special pairs in Proposition~\ref{p:specialpairs2}.

 %%%%%%%%%%%%%%%%%%%%%%%%%%%%%%%%%%%%
 \subsection{Characterizing when $\textrm{Deck}(f^k) \cong V_4$}

 In order to prove Theorem~\ref{t:DetectQuadraticCritPts}, we need to investigate exactly when we have $\Deck(f^k) \cong V_4$.
 
 We begin with a few preliminary lemmas.  We will use the notation $\textrm{Fix}(\phi)$ to denote the set of fixed points of a map $\phi$. 
  \begin{lemma}
 \label{Klein4lemma}
Let $f$ be a quadratic rational map so that $\textrm{Deck}(f^k) \cong V_4$ for some iterate $k\in\mathbb N$. Let $\phi\in \textrm{Deck}(f^k)$ be the generator of $\mathrm{Deck}(f)$. Then $\mathrm{Fix}(\phi)=\mathcal C_f$, and every element of $\mathrm{Deck}(f^k)$ maps the set $\mathrm{Fix}(\phi)$ to itself. 
%\begin{enumerate}
%\item all elements in $\textrm{Deck}(f^k)\setminus\{\mathrm{id}\}$ have exactly two fixed points, 
%\item for distinct $\phi, \psi \in  \textrm{Deck}(f^k)\setminus \{\mathrm{id}\}$, $\mathrm{Fix}(\phi)\cap \mathrm{Fix}(\psi)=\emptyset$, and  
%\item for $\phi, \psi \in \textrm{Deck}(f^k)$, $\phi$ maps $ \mathrm{Fix}(\psi)$ to itself. 
%\end{enumerate}
\end{lemma}
 \begin{proof} Write $\textrm{Deck}(f^k) =\{\eta,\psi,\phi,\mathrm{id}\}$ with $\textrm{Deck}(f)=\{\phi,\mathrm{id}\}$. Since $\phi$ is a deck transformation of $f$, it preserves fibers of $f$ (by Proposition  \ref{p:elementaryfacts}). Because $f$ is quadratic, the fiber over each critical value contains exactly one point (a critical point of $f$). Therefore, $\phi$ must fix each critical point of $f$. Since $\phi$  can have at most two fixed points, this implies $\mathrm{Fix}(\phi)=\mathcal C_f$. Write $\mathrm{Fix}(\phi)=\{c_1,c_2\}$, and consider how the elements $\psi$ and $\eta$ act on this set. Because  $\textrm{Deck}(f^k) \cong V_4$, we have 
 \begin{align*}
&x:=\psi(c_1)=\eta(c_1), \quad \psi(x)=\eta(x)=c_1,\quad \text{and}\\
&y:=\psi(c_2)=\eta(c_2), \quad \psi(y)=\eta(y)=c_2.
\end{align*}
So the involutions $\psi$  and $\eta$ coincide on the set $\{x,y,c_1,c_2\}$. Since $\eta$ and $\psi$ are distinct M\"obius transformations, we must have $x,y\in\{c_1,c_2\}$. The points $x$ and $y$ are distinct, so there are two possibilities: 
\[
x=c_1\text{ and }y=c_2 \qquad \text{or}\qquad  x=c_2 \text{ and }y=c_1.
\] 
In the first case, $\mathrm{Fix}(\phi)=\mathrm{Fix}(\psi)=\mathrm{Fix}(\eta)$. But then $\phi$, $\psi$, and $\eta$ would be three nontrivial involutions with the same pair of fixed points, so they would all coincide, which is not possible. In the second case, $c_1$ and $c_2$ comprise a common 2-cycle for the elements $\eta$ and $\psi$. In particular, we see that $\eta$ and $\psi$ map $\mathrm{Fix}(\phi)$ to itself. 
 \end{proof}
 
  \begin{lemma}\label{Tomlemma}
Let $f$ be a bicritical rational map with critical point set $\mathcal{C}_f$. Then if $\mu$ is a M\"obius transformation such that $\mu(\mathcal{C}_f) = \mathcal{C}_f$, then there exists a unique M\"obius transformation $\nu$ such that $\nu \circ f = f \circ \mu$. Furthermore $\nu (\mathcal{V}_f) = \mathcal{V}_f$.
\end{lemma}

\begin{proof}
Once we prove existence, the uniqueness will follow from the surjectivity of $f$. We first prove the existence result for $g(z) = z^d$. In this case, $\mu$ is a M\"obius transformation such that $\mu(\mathcal{C}_g) = \mathcal{C}_g$ if and only if $\mu = a z^{\pm 1}$ for some $a \in \mathbb{C} \setminus \{ 0 \}$. But then $g \circ \mu = a^d z^{\pm d}$, and so taking $\nu = a^d z$ completes the proof for $g(z) = z^d$.

Now suppose that $f$ is bicritical of degree $d$. Then there exist M\"obius transformations $\alpha$ and $\beta$ such that $f = \alpha \circ g \circ \beta$, where $g(z) = z^d$. In particular $\beta(\mathcal{C}_g) = \mathcal{C}_f$. Thus if $\mu$ fixes $\mathcal{C}_f$ as a set then $\mu' = \beta^{-1} \circ \mu \circ \beta$ fixes $\mathcal{C}_g$ as a set, and by the above there exists $\nu'$ such that $\nu' \circ g = g \circ \mu'$. Hence taking $\nu = \alpha \circ \nu' \circ \alpha^{-1}$ we see that
\begin{align*}
 \nu \circ f	 	&= \nu \circ (\alpha \circ g \circ \beta) \\
 			&= (\alpha \circ \nu' \circ \alpha^{-1}) \circ \alpha \circ g \circ \beta \\
			& = \alpha \circ (\nu' \circ g) \circ \beta \\
			&= \alpha \circ (g \circ \mu') \circ \beta \\
			&= \alpha \circ g \circ (\beta \circ \mu \circ \beta^{-1}) \circ \beta \\
			&= \alpha \circ g \circ \beta \circ \mu \\
			&= f \circ \mu
\end{align*}
as desired. The fact that $\nu (\mathcal{V}_f) = \mathcal{V}_f$ is clear.
\end{proof}

Observe that for a M\"{o}bius transformation  $\mu \in \Deck(f^k)$, for some $k>1$, then if $\nu$ is the whose existence is guaranteed by Lemma~\ref{Tomlemma}, we have following commutative diagram.
\begin{center}
  \begin{tikzcd} 
  \hat{\mathbb{C}} \arrow[r, "\mu"] \arrow[d, "f" ']
    &   \hat{\mathbb{C}} \arrow[d, "f" ] \\
      \hat{\mathbb{C}} \arrow[r, "\nu"] \arrow[d, "f^{k-1}" ']
    &   \hat{\mathbb{C}} \arrow[d, "f^{k-1}" ] \\
    \hat{\mathbb{C}} \arrow[r, "\mathrm{id}" ]
&   \hat{\mathbb{C}} \end{tikzcd}
\end{center}
The large outermost rectangle commutes since $\mu\in \textrm{Deck}(f^k)$. Therefore, the square in the bottom commutes as well. As a consequence, $\nu \in  \textrm{Deck}(f^{k-1})$.

 \begin{lemma}
 \label{l:minimalk}
 Let $f$ be a quadratic rational map.  If there exists $k \in \mathbb{N}$ such that $\textrm{Deck}(f^k) \cong V_4$, then the minimal such $k$ is $k=2$.  
 \end{lemma}
  
  \begin{proof}
Suppose that $k$ is minimal so that $\textrm{Deck}(f^k) \cong V_4$; note that $k>1$. Write $\textrm{Deck}(f^k) =\{\eta,\psi,\phi,\mathrm{id}\}$ with $\textrm{Deck}(f)=\{\phi,\mathrm{id}\}$. By Lemma \ref{Klein4lemma}, $\eta$ maps $\mathcal C_f$ to itself, so by Lemma \ref{Tomlemma}, there is a M\"obius transformation $\mu$ so that the top square in the following diagram commutes. 
 \begin{center}
  \begin{tikzcd} 
  \hat{\mathbb{C}} \arrow[r, "\eta"] \arrow[d, "f" ']
    &   \hat{\mathbb{C}} \arrow[d, "f" ] \\
      \hat{\mathbb{C}} \arrow[r, "\mu"] \arrow[d, "f^{k-1}" ']
    &   \hat{\mathbb{C}} \arrow[d, "f^{k-1}" ] \\
    \hat{\mathbb{C}} \arrow[r, "\mathrm{id}" ]
&   \hat{\mathbb{C}} \end{tikzcd}
\end{center}
As we saw above, this means that $\mu \in \Deck(f^{k-1})$. Note that since $k$ is minimal with respect to the property that $\textrm{Deck}(f^k) \cong V_4$, we must have $\textrm{Deck}(f^{k-1}) =\textrm{Deck}(f)$, so $\mu\in \mathrm{Deck}(f)$, and  $f = f \circ \mathrm \mu$. But then the outer rectangle commutes for $f^{k-1}=f$, or $k=2$ as desired. 
  \end{proof}

We now show that, with one (up to conjugacy) exception, if $\Deck(f^2) \cong V_4$ then $\Deck(f^k) \cong V_4$ for all $k \geq 2$.
 
 \begin{lemma} \label{l:criticallyCoalescing4DistinctPts}
 Let $f$ be a critically coalescing  quadratic rational map.  Then $\mathcal{V}_f \cup \mathcal{C}_f$ consists of $4$ distinct points. 
 \end{lemma}
 
 \begin{proof}
 Write $\mathcal{V}_f = \{v_1, v_2\}$ and $\mathcal{C}_f = \{c_1,c_2\}$.  All bicritical rational maps satisfy $c_1 \neq c_2$ and $v_1 \neq v_2$ (\cite{MilnorBicritical}).  Suppose, for a contradiction, that  $c_1 = v_j$ for $j = 1$ or $2$.  Then $v_j = f(c_1) = f(v_1) = f(v_2)$. But since $c_1$ maps forward with local degree $2$, we see that $v_j$ has at least three preimages (counting multiplicity) under $f$ which is impossible since $f$ is quadratic. 
 \end{proof}
 
% \begin{lemma} \label{l:quadraticsDeckGroupPossibilities}
% Let $f$ be a quadratic rational map that is not a power map.  Then every nontrivial element of $\Deck(f^2)$ has order $2$. 
% \end{lemma}
% 
% \begin{proof}
% Since $f$ is not a power map, $f^2$ has either $3$ or $4$ distinct critical values. First, consider the case that $f^2$ has 3 critical values, so that $f$ is  critically coalescing \color{red} This is not quite true, but I think we can prove if $f$ is not critically coalescing then $\Deck(f^2) \cong \mathbb{Z}_2$. \color{black}  Then the fiber over each of these critical values consists of two critical points of $f^2$.  By Proposition \ref{p:elementaryfacts} 
% parts \eqref{i:localdegreespreserved} and \eqref{i:DeckGroupFixesFibers}, any non-identity deck transformation $\tau \in \Deck(f^2)$ must transpose the elements of at least two of the pairs of these pairs of critical points, implying $\tau^2$ is the identity map.  
% 
% Next, consider the case that $f^2$ has 4 critical values. Two of these critical values, say $v_1$ and $v_2$ have degree partition $\{2,1,1\}$ under $f^2$ and the other two, say $v_3$ and $v_4$ have degree partition $\{2,2\}$ under $f^2$. If $\tau \in \Deck(f^2)$ is not the identity, then $\tau$ must fix the two critical points in the fibers over $v_1$ and $v_2$, and transpose the critical points in the fibers over $v_3$ and $v_4$. But then $\tau^2$ is the identity.   
% \end{proof}
%  
%  \color{red}Alternative to the above
 
 \begin{lemma}\label{l:quadraticsDeckGroupPossibilities}
  Let $f$ be a quadratic rational map. Then $\Deck(f^2) \cong \mathbb{Z}_4$ if and only if $f$ is a power map.
 \end{lemma}
 
 \begin{proof}
  It is clear that if $f$ is a power map then $\Deck(f^2) \cong \mathbb{Z}_4$. So suppose $\Deck(f^2) \cong \mathbb{Z}_4$, and let $\phi \in \Deck(f^2)$ have order $4$. We see that $\mathrm{Fix}(\phi) = \mathcal{C}_f = \{c_1,c_2\}$, and so by Lemma~\ref{Tomlemma} there exists $\mu$ such that $\mu \circ f = f \circ \phi$, and $\mu(\mathcal{V}_f) = \mathcal{V}_f$. As in the argument of Lemma~\ref{l:minimalk}, we see that $\mu \in \Deck(f)$, and since $\phi$ is not an element of $\Deck(f)$, it follows that $\mu$ is the unique order $2$ element of $\Deck(f)$. In particular, $\mathrm{Fix}(\mu) = \mathcal{C}_f$. Denoting $v_i = f(c_i)$, we observe that
  \[
   v_i = f(c_i) = f \circ \phi (c_i) = \mu \circ f (c_i) = \mu(v_i)
  \]
and so $\mathcal{C}_f = \mathrm{Fix}(\mu) = \mathcal{V}_f$. Thus $f$ is a power map. 
 \end{proof}

 It follows that if $f$ is not a power map, then $\Deck(f^2)$ must be isomorphic to either $\mathbb{Z}_2$ or $V_4$. In either case, every non-identity element of $\Deck(f^2)$ has order $2$.
 
%  \color{black}
Before proceeding, we need a well-known result about commuting M\"obius transformations (see e.g. \cite{Beardon:Groups}, Theorem 4.3.6).

\begin{lemma}\label{l:Beardon}
 Let $\phi$ and $\psi$ be non-identity M\"obius transformations with fixed point sets $\mathrm{Fix}(\phi)$ and $\mathrm{Fix}(\psi)$ respectively. Then the following are equivalent.
 \begin{enumerate}
  \item $\phi \circ \psi = \psi \circ \phi$
  \item $\phi(\mathrm{Fix}(\psi)) = \mathrm{Fix}(\psi)$ and $\psi(\mathrm{Fix}(\phi)) = \mathrm{Fix}(\phi)$.
  \item Either $\mathrm{Fix}(\psi) =\mathrm{Fix}(\phi)$ or $\phi$, $\psi$ and $\phi \circ \psi$ are involutions and $\mathrm{Fix}(\phi) \cap \mathrm{Fix}(\psi) = \varnothing$. 
 \end{enumerate}
 \end{lemma}
 
 \begin{proposition}\label{p:characterizecritcoal}
 Let $f$ be a quadratic rational map that is not a power map.  Then the following are equivalent:
 \begin{enumerate}
 \item \label{i:2ndV4} $\textrm{Deck}(f^2) \cong V_4$,
 \item \label{i:kthV4} $\textrm{Deck}(f^k) \cong V_4$ for some $k \in \mathbb{N}$,
 \item  \label{i:critcoal} $f$ is critically coalescing. 
 \end{enumerate}
 \end{proposition}
 
 \begin{proof}
  Lemma \ref{l:minimalk} gives \eqref{i:2ndV4} if and only if \eqref{i:kthV4}.  We will prove the equivalence of conditions  \eqref{i:2ndV4} and \eqref{i:critcoal}.
 Suppose $\textrm{Deck}(f^2) \cong V_4$. Then $\textrm{Deck}(f)$ is a cyclic group of order $2$ (by Lemma \ref{l:minimalk}) generated by a rotation $\mu$ about the axis connecting the critical points of $f$ (by Lemma \ref{l:elliptic}). Consider any element $\tau \in \textrm{Deck}(f^2) \setminus \textrm{Deck}(f)$. We claim that the set $\mathcal{C}_f$ must be fixed by $\tau$. Indeed, since $V_4$ is abelian, we see that $\tau$ commutes with $\mu$, the unique non-identity element of $\Deck(f)$. Since $\mathrm{Fix}(\mu) = \mathcal{C}_f$, it follows from Lemma~\ref{l:Beardon} that $\tau(\mathcal{C}_f) = \mathcal{C}_f$.
 
%  To obtain a contradiction assume that $f$ is not critically coalescing. Then $\rho_{f^2}(f(v_i))$ contains only one critical point of $f$, call it $c_i$, which maps forward under $f^2$ with local degree $>1$, and two points which map forward under $f^2$ with local degree $1$.  Hence Proposition \ref{p:elementaryfacts} part \eqref{i:localdegreespreserved} implies $\tau$ fixes $c_i$.  
 
%  Since $\tau$ has order $2$ and all M\"obius transformations of finite order are elliptic, we see that $\tau$ is an order $2$ elliptic rotation. 
 If $\tau$ fixed $\mathcal{C}_f$ pointwise, then $\tau$ would coincide with the generator of $\textrm{Deck}(f)$, a contradiction.  Hence $\tau$ must interchange the two points of $\mathcal{C}_f$.  Then Proposition \ref{p:elementaryfacts} part \eqref{i:DeckGroupFixesFibers} implies the two critical points of $f$ belong to the same fiber under $f^2$, i.e. $f$ is critically coalescing.
 
Now suppose $f$ is critically coalescing. By Lemma \ref{l:criticallyCoalescing4DistinctPts}, the critical points and values of $f$, which we will denote $c_1,c_2$ and $v_1,v_2$ respectively, are all distinct.  We may therefore normalize $f$ so that $c_1=0$, $c_2 = \infty$ and $v_2 = 1$, which means that $f$ belongs to the one-parameter family given by 
\begin{equation}\label{e:CritCoalfam}
f_a(z) = \frac{z^2-a}{z^2+a}.
\end{equation}
As the reader may verify, the maps $z \mapsto -z$ and $z \mapsto \tfrac{a}{z}$ belong to $\textrm{Deck}(f^2)$ and they generate a subgroup isomorphic to $V_4$.  Therefore, by Lemma \ref{l:quadraticsDeckGroupPossibilities} and Proposition \ref{p:elementaryfacts} part \eqref{i:possibilities}, $\textrm{Deck}(f^2) \cong V_4$.\end{proof}

\begin{proposition}\label{p:specialpairs2}
Let $f$ be a critically coalescing quadratic rational map with critical points $c_1$ and $c_2$. Then the special pairs of $f$ are the sets $\mathcal{C}_f$, $f^{-1}(c_1)$ and $f^{-1}(c_2)$.
\end{proposition}

\begin{proof}
Denote $f^{-1}(c_1) = \{a_1,a_2\}$, $f^{-1}(c_2) = \{b_1,b_2\}$. From Proposition~\ref{p:characterizecritcoal}, we have $\Deck(f^2) \cong V_4$. We know that the elements of $\mathcal{C}_f$ are fixed by the unique non-identity element $\mu \in \Deck(f) \subset \Deck(f^2)$. Now observe that $f^{-2}(v_1) = \{a_1,a_2\}$. Since elements of $\Deck(f^2)$ preserve the fibers under $f^2$, we see that orbit of $a_1$ under the action of $\Deck(f^2)$ contains at most two elements. By the Orbit-Stabilizer Theorem, the stabilizer of $a_1$ must contain a non-identity element $\nu$ of $\Deck(f^k)$. In this case we must have $\nu(a_2) = a_2$, so the fixed points of $\nu$ are precisely the elements of $f^{-1}(c_1)$. The case for the elements of $f^{-1}(c_2)$ is similar.  
\end{proof}
    
We complete this subsection by showing that if $f$ is a quadratic rational map such that $\Deck(f^k)$ is dihedral for some $k$, then $f$ is critically coalescing. First we strengthen the result of Lemma~\ref{l:quadraticsDeckGroupPossibilities}.

\begin{lemma}\label{l:powermaps}
Let $f$ be a quadratic rational map. Then $\Deck(f^k)$ is a cyclic group of order greater than $2$ if and only if $f$ is a power map.
\end{lemma}

\begin{proof}
It is clear that if $f$ is a power map then $\Deck(f^k)$ is a cyclic group of order $2^k$. Let $n \geq 2$, and suppose that $k$ is minimal such that $\Deck(f^k)$ is a cyclic group of order $2^n$. Let $\nu$ be a generator of $\Deck(f^k)$. Then $\mu = \nu^{2^{n-1}}$ has order $2$ and so belongs to $\Deck(f)$. In particular we have $\mathrm{Fix}(\phi) = \mathcal{C}_f$ for all non-identity elements in $\Deck(f^k)$. Thus we may inductively apply the argument following Lemma~\ref{Tomlemma} to see that the following diagram commutes.
\begin{center}
 \begin{tikzcd} 
 \hat{\mathbb{C}} \arrow[r, "\nu"] \arrow[d, "f^{k-1}" ']
   &   \hat{\mathbb{C}} \arrow[d, "f^{k-1}" ] \\
     \hat{\mathbb{C}} \arrow[r, "\mu"] \arrow[d, "f" ']
   &   \hat{\mathbb{C}} \arrow[d, "f" ] \\
   \hat{\mathbb{C}} \arrow[r, "\mathrm{id}" ]
&   \hat{\mathbb{C}} \end{tikzcd}
\end{center}
Furthermore, it follows that $\mu$ must map $\mathcal{V}_f$ to itself. Let $\mathcal{V}_f = \{v_1, v_2 \}$. If $\mu$ interchanges the elements of $\mathcal{V}_f$, then we would have $f(v_1) = f(v_2)$, and so by Proposition~\ref{p:characterizecritcoal}, $\Deck(f^k)$ would contain a subgroup isomorphic to $V_4$; a contradiction. So we see that $\mu$ must fix the elements of $\mathcal{V}_f$ pointwise. Since $\mu \in \Deck(f)$, the fixed points of $\mu$ are also the critical points of $f$. Hence $\mathcal{C}_f = \mathcal{V}_f$, and $f$ is a power map.
\end{proof}

%\begin{lemma}\label{l:Decksize}
%Let $f$ be a quadratic rational map. Then $\frac{ |\Deck(f^k) |}{|\Deck(f^{k-1})| } \in \{1,2\}$.
%\end{lemma}
%
%\begin{proof}
%If $\Deck(f^k) = \Deck(f^{k-1})$, we are done. Write $\Deck^\ast(f^k) = \Deck(f^k) \setminus \Deck(f^{k-1})$ and suppose there exists $\phi \in \Deck^\ast(f^k)$. Since $\phi$ is not an element of $\Deck(f^{k-1})$, then if $\mu$ is the unique order $2$ element of $\Deck(f)$, the following diagram commutes.
%\begin{center}
% \begin{tikzcd} 
% \hat{\mathbb{C}} \arrow[r, "\phi"] \arrow[d, "f^{k-1}" ']
%   &   \hat{\mathbb{C}} \arrow[d, "f^{k-1}" ] \\
%     \hat{\mathbb{C}} \arrow[r, "\mu"] \arrow[d, "f" ']
%   &   \hat{\mathbb{C}} \arrow[d, "f" ] \\
%   \hat{\mathbb{C}} \arrow[r, "\mathrm{id}" ]
%&   \hat{\mathbb{C}} \end{tikzcd}
%\end{center}
%In particular note that $\mu \circ f^{k-1} = f^{k-1} \circ \phi$. It follows there exists a homomorphism $h \colon \Deck(f^k) \to \Deck(f^{k-1})$, defined by $h(\phi) = \mu$ when $\mu \in \Deck^\ast(f^k)$ and $h(\psi) = \mathrm{id}$ when $\psi \in \Deck(f^{k-1})$. We see that the kernel of $h$ is $h^{-1}(\mathrm{id}) = \Deck(f^{k-1})$. Furthermore, $h^{-1}(\mu) = \Deck^\ast(f^k)$, and a standard property of homomorphism states that all preimages under a homomorphism contain the same number of elements. Thus $|\Deck^\ast(f^k)| = |\Deck(f^{k-1})|$ and the result follows.
%\end{proof}

The following strengthens the result of Lemma~\ref{l:minimalk}.

\begin{proposition}\label{p:dihedralimpliescoalescing}
Let $f$ be a quadratic rational map. Then if $\Deck(f^k)$ is dihedral for some $k$, then $f$ is critically coalescing.
\end{proposition}

\begin{proof}
Write $\Deck(f) = \{ \mathrm{id}, \mu \}$ and suppose $k>1$ is minimal such that $\Deck(f^k)$ is dihedral. Let $\Gamma$ be a subgroup of $\Deck(f^k)$ such that $\Gamma \cong V_4$ and $\Deck(f) \subseteq \Gamma$. Write $\Gamma = \{ \mathrm{id}, \mu, \alpha, \beta \}$. Since $\Gamma$ is abelian and $\mathrm{Fix}(\mu) = \mathcal{C}_f$, then $\alpha(\mathcal{C}_f) = \beta(\mathcal{C}_f) = \mathcal{C}_f$. Thus, by Lemma~\ref{Tomlemma} and the subsequent discussion, there exists $\nu \in \Deck(f^{k-1})$ such that $\nu \circ f = f \circ \alpha$. Furthermore,  $\nu(\mathcal{V}_f ) = \mathcal{V}_f = \{v_1, v_2 \}$, and since $\alpha$ is not an element of $\Deck(f^2)$, we have $\nu \neq \mathrm{id}$ . By the assumption on the  minimality of $k$, $\Deck(f^{k-1})$ must be cyclic, and so for any non-identity elements $\gamma \in \Deck(f^{k-1})$ we have $\mathrm{Fix}(\gamma) = \mathrm{Fix}(\mu) = \mathcal{C}_f$. Thus $\mathrm{Fix}(\nu) = \mathcal{C}_f$.

If $\nu$ fixes the elements of $\mathcal{V}_f$ pointwise, then we have $\mathcal{C}_f = \mathcal{V}_f$, and so $f$ is a power map. But this is impossible, since $\Deck(f^k)$ is always cyclic for power maps. So $\nu$ must be an involution which exchanges the elements of $\mathcal{V}_f$. But since $\Deck(f^{k-1})$ is cyclic, it contains a unique involution, namely $\mu$, and so $\nu = \mu$. Thus $\mu \in \Deck(f)$ interchanges the elements of $\mathcal{V}_f$, and so $f(v_1) = f(v_2)$.
\end{proof}

 \subsection{Remarks on Critically Coalescing Quadratic Rational Maps}
  
Consider the family $f_a$ (see Figure~\ref{f:critcoal}) from equation \eqref{e:CritCoalfam}. The authors have not been able to find any reference to this family in the literature. Accordingly, we prove some preliminary results about this family here, and leave a more detailed investigation for future study. Note that the maps $f_a$ and $f_{-1/a}$ are conjugate via the map $\phi(z) \mapsto -1/z$. 
 \begin{figure}[ht]
\includegraphics[width=0.8\textwidth]{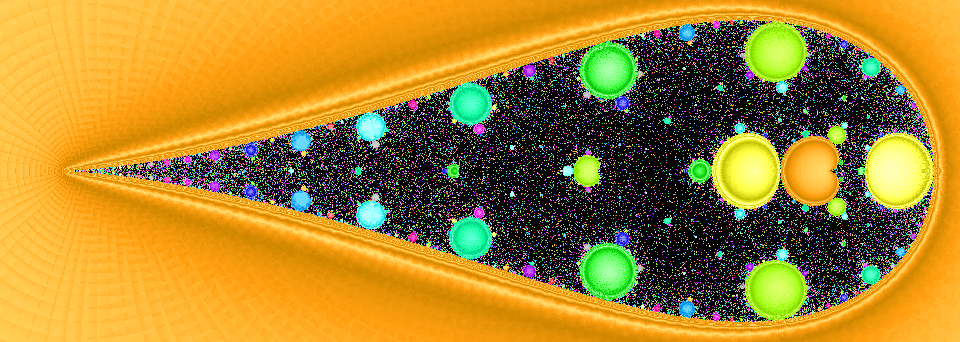}
\caption{The parameter space of the family $f_a(z) = \frac{z^2-a}{z^2+a}$, characterized by the equivalent conditions that $\Deck(f^2) \cong V_4$ and $f(v_1) = f(v_2)$. The colors of the hyperbolic components represent the period of the attracting orbit of the map. For example, orange is period $1$, yellow is period $2$.} 
\label{f:critcoal}
\end{figure} 

For the moment we show that for all $a \neq \pm 1$, we have $\Deck(f_a^k) \cong V_4$ for all $k \geq 2$. First we need a special case of a result of Pakovich (\cite{Sym}, Theorem 5.2). Given a rational map $F$, we define $\Deck_\infty(F) = \bigcup_{k=1}^{\infty} \Deck(F^k)$.

\begin{proposition}[\cite{Sym}]\label{p:Pakovich}
Let $F$ be a quadratic rational map. If $\sigma \in \Deck_\infty(F)$ then $F \circ \sigma = \beta \circ F \circ \beta^{-1}$ for some M\"obius map $\beta$.
\end{proposition}

%\begin{proof}
%(This was proven in \cite{Sym}, in far more generality. But the proof uses ideas about irreducible curves; it would be nice if there was a ``purely dynamical'' proof we can write).
%\end{proof}

\begin{proposition}\label{p:critcoaldecks}
Let $f_a(z) = \frac{z^2-a}{z^2 + a}$.
\begin{enumerate}
\item If $a = \pm 1$ then $\Deck(f^2) \cong V_4$ and $\Deck(f^k) \cong D_8$ for all $k \geq 3$.
\item If $a \neq \pm 1$, then $\Deck(f^k_a) \cong V_4$ for all $k \geq 2$. 
\end{enumerate} 
\end{proposition}

\begin{proof}\mbox{}
\begin{enumerate}
\item This was proven by Pakovich in \cite{Sym}. Our method to prove the second part is follows the method of Pakovich.
\item We know that $\Deck(f_a^2) = \{ \mathrm{id}, -z, a/z, -a/z \} \cong V_4$. Since we know that $\Deck(f_a^k)$ cannot be a polyhedral group, then if $\Deck(f^k_a)$ ($k > 2$) is not isomorphic to $V_4$, it must be isomorphic to a dihedral group. Such a dihedral group must contain an element of order greater than $2$ which, by Lemma~\ref{l:nonquadraticcandetectcps}, fixes the critical points $0$ and $\infty$, and so this element must be of the form $\sigma(z) = cz$ for some $c \in \mathbb{C}^\ast$. By Proposition~\ref{p:Pakovich}, we need to show that if $\sigma(z) = c z$ and
\begin{equation}\label{e:Pakovich}
f_a \circ \sigma = \beta \circ f_a \circ \beta^{-1}
\end{equation}
for some M\"obius map $\beta$ then $c = \pm 1$. By equation \eqref{e:Pakovich}, since both sides of the equation have the same critical points, any $\beta$ satisfying the equation must be of the form $\beta(z) = dz^{\pm 1}$. If $\beta(z) = dz$, then \eqref{e:Pakovich} becomes
\[
 \frac{c^2 z^2 - a}{c^2 z^2 + a} = \frac{1}{d}\frac{d^2 z^2 - a}{d^2 z^2 + a}
\]
which is solved by $c = \pm 1$ and $d = 1$. On the other hand, if $\beta = d/z$, then \eqref{e:Pakovich} now becomes
\[
 \frac{c^2 z^2 - a}{c^2 z^2 + a} = d \frac{ad^2 + z^2}{ad^2-z^2},
\]
and this is solved by $d=-1$ and $c = \pm ai$.

To complete the proof, observe that if $\sigma(z) =  ai z$ is an element of $\Sigma_\infty(F)$, then so is $\sigma^2(z) = -a^2 z$. But the above computations show that this would mean $-a^2 \in \{ 1,-1, ai, -ai \}$. This yields the possibilities $a \in \{ i, -i , 1, -1 \}$. Since by assumption $a \neq \pm 1$ we are left with the case $a = \pm i$. But then we would have $a i = \mp 1$, means $\sigma(z) = \pm z$. 
\end{enumerate}
\end{proof}

We note that no hyperbolic map of the form $f_a$ can be a mating (see e.g \cite{MilnorMating,ShishikuraTanMating}); a hyperbolic map which is a mating has to have disjoint critical orbits, and so the condition $f(v_1) = f(v_2)$ is incompatible with this requirement. However, there do exist matings in this family. For example, the map $f_i$ is equal to the self-mating of the quadratic polynomial which is the landing point of the parameter ray of argument $1/4$ in the Mandelbrot set.

\section{Detecting critical points and values of non-critically coalescing quadratic maps}\label{s:DetectCritnoncritcoal}
Before addressing the subtler critically coalescing (or, equivalently by Proposition \ref{p:characterizecritcoal}, $\textrm{Deck}(f^k) \cong V_4$) case,  we briefly show that the techniques of Section \ref{s:higherdegree} can be used to detect $\mathcal{C}_f$ and $\mathcal{V}_f$ in the non-critically coalescing quadratic case.  First we show we can detect the critical points of $f$.

\begin{lemma}
Let $f$ be a quadratic rational map which is not critically coalescing and $k \geq 1$. Then $\Deck(f^k)$ contains an element of order greater than $2$ or $\Deck(f^k) = \Deck(f) \cong \mathbb{Z}_2$.
\end{lemma}

\begin{proof}
Suppose that all elements of $\Deck(f^k)$ have order $1$ or $2$. If $\Deck(f^k) \cong V_4$ then by Proposition~\ref{p:characterizecritcoal} $f$ would be critically coalescing, which is a contradiction. Thus $\Deck(f^k) \cong \mathbb{Z}_2$.
\end{proof}

\begin{corollary}\label{cor:detectcpsnocritcoal}
Let $f$ be a quadratic rational map which is not critically coalescing and $k \geq 1$. Then we can detect the critical points of $f$ from $\mathcal{C}_f$.
\end{corollary}

\begin{proof}
By the lemma, either $\Deck(f^k)$ contains an element of order greater than $2$ or $\Deck(f^k) \cong \mathbb{Z}_2$. If $\Deck(f^k)$ contains an element $\mu$ of order greater than $2$, then by Lemma~\ref{l:powermaps} and Proposition~\ref{p:dihedralimpliescoalescing}, $f$ is a power map and the critical points of $f$ are fixed by every non-identity element of $\Deck(f^k)$. On the other hand if $\Deck(f^k) \cong \mathbb{Z}_2$ then $\Deck(f^k) = \Deck(f)$ and so the critical points of $f$ are fixed by the unique non-identity element of $\Deck(f^k)$.
\end{proof}

%In the quadratic non-critically coalescing case, Lemma~\ref{l:nontypealphacandetectcvs} shows gives an easy method for detecting the critical values from the iterate (the analogues of  Lemma \ref{l:nonquadraticcandetectcvs}). 
%While Lemma \ref{l:nonquadraticcandetectcvs} shows that one can detect the critical values from the iterate in the nonquadratic case,  Lemmas~\ref{l:nontypealphacandetectcvs} and \ref{l:coalescing3pts} provide analogs in the quadratic case -- with the difference that in the critically coalescing case they detect the 3-point set consisting of the two critical values and their common image.  

We can also easily detect the critical values in the non-critically coalescing case.

\begin{lemma} \label{l:nontypealphacandetectcvs}
Let $f$ be a quadratic rational map that is not critically coalescing. Then for each $k \in \mathbb{N}$,  $x \in \mathcal{V}_f$ if and only if $f^{-k}(x) \subseteq \mathcal{C}_{f^{k}}$.
\end{lemma}

\begin{proof}
The argument proceeds similarly to that for Lemma \ref{l:nonquadraticcandetectcvs}. This time, we notice that since $f$ is not critically coalescing, then for any $y \notin \mathcal{V}_f$, there exists an element $y' \in f^{-1}(y)$ which is not a critical value of $f$. Thus as before we may construct a sequence $x_0 = x,\dotsc,x_k$, so that $x_k \in f^{-k}(x)$ but $x_k$ is not a critical point of $f^{k}$.
\end{proof}

Lemma~\ref{l:nontypealphacandetectcvs} has the following immediate corollary.  

\begin{corollary} \label{l:cvscoincide}
Fix a rational map $F$.  If there exists a quadratic rational map $f$ that is not critically coalescing such that $f^k = F$ for some $k \in \mathbb{N}$, then 
$\mathcal{V}_f = \{x \in \hat{\mathbb{C}} \mid F^{-1}(x) \subseteq \mathcal{C}_F\}$. In particular we can detect the critical values of $f$ from $f^k$.
%Let $f$ and $g$ bicritical quadratic rational maps that are not Type $\alpha$ and suppose such that $f^k = g^k$ for some $k\in \mathbb{N}$.  
% Then $\mathcal{V}_f = \mathcal{V}_g$. 
\end{corollary}

%In the critically coalescing case, we are not able to immediately read off the critical values of $f$ from the fibers over the critical values of $f^k$. 

%%%%%%%%%%%%%%%%%%%%%%%%%%%%%%%%%%%
\section{Detecting Critical Points and Critical Values for Critically Coalescing quadratic rational maps}\label{s:Detectcritcoal}

In this section we discuss how we may detect $\mathcal{C}_f$ and $\mathcal{V}_f$ in the case where $\Deck(f^k) \cong V_4$. Along with the results in the previous section, this will allow us to complete the proof of Theorem~\ref{t:DetectQuadraticCritPts}.

Suppose $f$ is critically coalescing. By Proposition~\ref{p:specialpairs2}, there exist three special pairs
\[
f^{-1}(c_1) \coloneqq \{a_1,a_2\}, \quad  f^{-1}(c_2) \coloneqq \{b_1,b_2\}
\]
and the true critical points $\mathcal{C}_f = \{c_1,c_2 \}$, which are the fixed points of some non-identity element of $\Deck(f^k)$. \emph{A priori} we cannot distinguish these pairs from one another, but we do know the true critical points are one of these pairs. We will show that a deeper analysis of $f^k$ will allow us to differentiate $\mathcal{C}_f = \{c_1,c_2\}$ from the other pairs, thus allowing us to detect $\mathcal{C}_f$ from $f^k$. The following Lemma is immediate.

\begin{lemma}\label{l:critpreimscommonimage}
 Let $f$ be a critically coalescing quadratic rational map and $k \geq 3$. Then $$f^k(a_1) = f^k(a_2) = f^k(b_1) = f^k(b_2).$$
\end{lemma}

Recall that given a rational map $f$ with critical point set $\mathcal{C}_f$, the postcritical set of $f$ is the set
\[
 P_f = \bigcup_{i=1}^\infty f^i(\mathcal{C}_f).
\]
We assume in the following that $k>1$ is fixed.

\subsection{The case where $P_f$ does not contain a fixed point}

\begin{lemma}\label{l:critcoalnofixedpoint}
Let $f$ be a critically coalescing quadratic rational map and suppose that $P_{f}$ does not contain a fixed point of $f$. Then $f^k(c_1) \neq f^k(a_1)$.
\end{lemma}

\begin{proof}
 Since the postcritical set of $f$ does not contain a fixed point, we have for all $k \geq 3$ that $$f^k(c_1) = f^k(f(a_1)) = f(f^k(a_1)) \neq f^k(a_1).$$
\end{proof}

\begin{corollary}\label{cor:detectcpscritcoalnofixed}
 Let $f$ be a critically coalescing quadratic map and suppose that $P_{f}$ does not contain a fixed point of $f$. Then the critical points of $f$ are characterized uniquely by the following two properties.
 \begin{enumerate}
  \item The critical points of $f$ are fixed by a non-identity element of $\Deck(f^k)$.
  \item The critical points of $f$ have the same image under $f^k$, but this image is distinct from the image of elements of the other special pairs.
 \end{enumerate}
 In particular, we can detect the critical points of $f$ from $f^k$.
\end{corollary}

\begin{proof}
 The first claim is Lemma~\ref{l:specialpairs} and the second is Lemma~\ref{l:critcoalnofixedpoint}. \end{proof}
 
 A similar argument allows us to detect the critical values of $f$ in this case. We write $w = f(v_1) = f(v_2)$.
 
\begin{lemma}\label{l:critvalsdistinctimages}
 Let $f$ be critically coalescing and suppose that $P_{f}$ does not contain a fixed point of $f$. Then $f^k(v_1) \neq f^k(w)$.
\end{lemma}

\begin{proof}
 Since $w = f(v_1)$ we have $f^k(w) = f^k(f(v_1)) = f(f^k(v_1))$. Since $f^k(v_1)$ is not a fixed point of $f$, we see that $f^k(v_1) \neq f(f^k(v_1))$ and the result follows.
\end{proof}

\begin{corollary}\label{cor:detectcvscritcoalnofixed}
Let $f$ be critically coalescing and suppose that $P_{f}$ does not contain a fixed point of $f$. Then the critical values of $f$ are characterized uniquely by the following two properties.
\begin{enumerate}
\item The fiber above the critical values of $f$ under $f^k$ consists of only critical points of $f^k$.
\item The critical values of $f$ have the same image under $f^k$, but this image is distinct the image of other $w = f(v_1) = f(v_2)$, which is the other point whose fiber under $f^k$ contains only critical points of $f^k$.
\end{enumerate}
\end{corollary}

\begin{proof}
The first claim is from is Lemma~\ref{l:coalescing3pts} and the second claim is Lemma~\ref{l:critvalsdistinctimages}.
\end{proof}

\subsection{The case where $P_f$ contains a fixed point}

It only remains to show we can distinguish the true critical points of $f$ when $\Deck(f^k) \cong V_4$ and $P_f$ contains a fixed point. Our strategy is as follows. Since $f$ is a quadratic rational map, we know that $\Deck(f) \cong \Z_2$. Furthermore, the two fixed points of the non-identity element $\mu$ of $\Deck(f)$ are the critical points of $f$. Thus, it suffices to distinguish $\mu$ from the other elements of $\Deck(f^k) \cong V_4$.  Since $f(v_1) = f(v_2)$, the fixed point in $P_f$ must be unique. Let $\alpha$ be this fixed point and let $m$ be minimal such that $\alpha = f^m(v_1) = f^m(v_2)$. The assumptions on $f$ mean that all critical points of $f^k$ are simple and that $|V_{f^k}| = m+2$, with the following dynamics under $f$ (we denote $\beta_{j} = f^j(v_1) = f^j(v_2)$ for $1 \leq j < m$).
\[
\xymatrix{
c_1\ar@{->}[r]^{2:1} & v_1\ar[dr]  & \\
& &\beta_1 \ar[r] & \beta_2 \ar[r] & \dotsc \ar[r] & \beta_{m-1} \ar[r] & \alpha \ar@(ur,dr) \\
c_2\ar@{->}[r]^{2:1}  & v_2\ar[ur]  & 
}
\]
The following Lemma and its Corollary show that it suffices to be able to detect the elements $\alpha$ and $\beta_{m-1}$.

\begin{lemma}
Let $f$ be a quadratic rational map such that $\Deck(f^k) \cong V_4$. If $P_f$ contains a fixed point $\alpha$, then the non-identity element $\mu$ of $\Deck(f)$ is the unique element of $\Deck(f^k)$ such that $\mu(\alpha) = \beta_{m-1}$.
\end{lemma}

\begin{proof}
The assumption on $f$ means that $\alpha$ is not an element of one of the special pairs of $f$. Thus the orbit of $\alpha$ under the action of $\Deck(f^k) \cong V_4$ consists of four elements, and each element in this orbit is the image of $\alpha$ for a unique element of $\Deck(f^k)$. Since $f^{-1}(\alpha) = \{ \alpha, \beta_{m-1} \}$ and elements of the deck group are fiber-preserving, we see that if $\mu$ is the unique non-identity element of $\Deck(f)$ then $\mu(\alpha) = \beta_{m-1}$.
\end{proof}

\begin{corollary}\label{cor:detectalphabeta}
If we can detect the elements $\alpha$ and $\beta_{m-1}$, then we can detect $\mathcal{C}_f$ and $\mathcal{V}_f$.
\end{corollary}

\begin{proof}
As noted above, if we know the points $\alpha$ and $\beta_{m-1}$ we can recover the unique non-identity element $\mu$ of $\Deck(f)$ since it is characterized by the property that $\mu(\alpha) = \beta_{m-1}$. We can then detect the elements of $\mathcal{C}_f$, since they are precisely the fixed points of this element $\mu$. To detect the critical values, we see that by virtue of Lemma~\ref{l:coalescing3pts}, we can narrow down the options for $\mathcal{V}_f$ to the elements of the set $\{ v_1, v_2, \beta_1 \}$. But since $f(v_1) = f(v_2)$, we see that, using the deck transformation $\mu$ found above, $\mu(v_1) = v_2$ and $\mu(v_2) = v_1$. This allows us to distinguish $v_1$ and $v_2$ from $\beta_{1}$, and so we can detect the set $\mathcal{V}_f$.
\end{proof}

When $m>2$, it is possible to detect the point $\beta_{m-1}$ using purely combinatorial arguments. However, the case $m=2$, which we will deal with first, requires some further work. In this case, $P_f = \{ v_1, v_2, \beta, \alpha \}$ and $f$ is a quadratic Latt\`{e}s map with the following critical portrait.
\[
\xymatrix{
c_1\ar@{->}[r]^{2:1} & v_1\ar[dr]  & \\
& &\beta \ar[r] & \alpha \ar@(ur,dr) \\
c_2\ar@{->}[r]^{2:1}  & v_2\ar[ur]  & 
}
\]
Recall that the cross-ratio $[v_1:v_2:\alpha:\beta]$ is $z(\beta)$ where $z:\hatC \to \hatC$ is the global coordinate which satisfies $z(v_1)=0$, $z(v_2)=\infty$ and $z(\alpha)=1$. 
%Equivalently we have
%\[
% 	[v_1:v_2:\alpha:\beta] = \frac{(\alpha - v_1)(\beta - v_2)}{(\alpha-v_2)(\beta-v_1)}
%\]

\begin{lemma}\label{lemma:crossratio}
Assume $f$ is a quadratic rational map with critical values $v_1$ and $v_2$ satisfying $f(v_1)=f(v_2) = \beta$, $f(\beta)=\alpha$ and $f(\alpha) = \alpha$ with $\alpha \neq \beta$. Then, 
\[[v_1:v_2:\alpha:\beta]\in \left\{-1,3\pm 2\sqrt{2}\right\}.\]
\end{lemma}

\begin{proof}
Let us assume that the critical points of $f$ are $c_1$ and $c_2$ with associated critical values $v_1= f(c_1)$ and $v_2 = f(c_2)$. Let $w:\hatC \to \hatC$ and $z:\hatC \to \hatC$ be the global coordinates defined by 
\[w(c_1) = 0, \quad w(c_2) = \infty,\quad  w(\alpha)=1,\quad z(v_1) = 0,\quad z(v_2) = \infty \text{ and } z(\alpha)=1.\]
Then, $z\circ f$ and $w^2$ both have a double zero at $c_1$ and a double pole at $c_2$, and both take the value $1$ at $\alpha$. It follows that $z\circ f = w^2$ (note that here we use $w^2$ to denote the square of $w$, not the second iterate, as is the case in the rest of the paper). 

As a consequence,  setting $\kappa = w(v_1)$, we get
\[\kappa^2 = w^2(v_1) = z\circ f(v_1) = z(\beta) = z\circ f(v_2) = w^2(v_2),\]
so that $w(v_2)=-\kappa$. Since $w$ sends $(v_1,v_2,\alpha)$ to respectively $(\kappa,-\kappa,1)$ and $z$ sends $(v_1,v_2,\alpha)$ to respectively $(0,\infty,1)$, we have that 
\[z = \frac{(\kappa+1)(\kappa-w)}{(\kappa-1)(\kappa+w)}.\]
Thus
\[\kappa^2 = z(\beta) =\frac{(\kappa+1)(\kappa-w(\beta))}{(\kappa-1)(\kappa+w(\beta))}= \left(\frac{\kappa+1}{\kappa-1}\right)^2.\]
This forces \[\kappa\in \left\{\pm i,1\pm\sqrt{2}\right\} \text{ and } z(\beta) = \kappa^2 \in \left\{-1, 3\pm2\sqrt{2}\right\}.\qedhere\]
\end{proof}

%\color{red}[Need to finish the above; maybe add a little more detail, like the equation for the cross-ratio? I have put an expression for the cross-ratio as a comment, but I need to check it matches what is used here-Tom]\color{black}

\begin{lemma}\label{l:detectcpscvsmequals2}
 Let $f$ be critically coalescing with $f(v_1)=f(v_2) = \beta$, $f(\beta)=\alpha$ and $f(\alpha) = \alpha$ with $\alpha \neq \beta$. If $g$ is a rational map such that $f^k = g^k$ for some $k \geq 1$, then $\mathcal{V}_f = \mathcal{V}_g$ and $\mathcal{C}_f = \mathcal{C}_g$. In particular, we can detect $\mathcal{C}_f$ and $\mathcal{V}_f$ from $f^k$.
\end{lemma}

\begin{proof}
 By Corollary~\ref{cor:detectalphabeta}, we need to show we can recover the elements $\alpha$ and $\beta$. Firstly, by Lemma~\ref{l:coalescing3pts}, we see that $\alpha$ is the unique element of $P_f$ for which $\rho_{f^k}(\alpha)$ contains points which are not critical points of $f^k$. To recover $\beta$, again note that following Lemma~\ref{l:coalescing3pts}, the set $\mathcal{V}_g$ consists of two points from the triple $S = \{ v_1, v_2 , \beta \}$. By Lemma~\ref{lemma:crossratio}, we know that
 \[
  [v_1:v_2:\alpha:\beta]\in \left\{-1,3\pm 2\sqrt{2}\right\}.
 \]
If $\mathcal{V}_g \neq \mathcal{V}_f$, then either $\mathcal{V}_g = \{ v_1, \beta \}$ or $\mathcal{V}_g = \{ v_2, \beta \}$.

In the first case we would have
\[
 [\beta : v_1 : \alpha : v_2 ] \in \left\{ \frac{1}{2} , \frac{1\pm \sqrt{2}}{2} \right\} 
\]
and in the second case, we have
\[
 [\beta : v_2 : \alpha : v_1] = \left\{ \frac{1}{2} , \frac{1\pm \sqrt{2}}{2} \right\}.
\]
In either case $g$ would contradict the conclusion of Lemma~\ref{lemma:crossratio}. Thus $\mathcal{V}_f = \mathcal{V}_g$, and so we can recover $\beta$ as the unique element of $S$ which does not belong to $\mathcal{V}_f$. Thus we can detect $\mathcal{C}_f$ and $\mathcal{V}_f$.
\end{proof}

Now assume that $m > 2$. We now show that we can detect the critical points of $f$ from $f^k$.

\begin{lemma}\label{l:notneeded}
If $k \leq m+1$ then we can detect the critical points and critical values of $f$ from $f^k$.
\end{lemma}

\begin{proof}
This is essentially the same as Corollaries~\ref{cor:detectcpscritcoalnofixed} and \ref{cor:detectcvscritcoalnofixed}, since in this case we have $f^k(c_1) \neq f^k(a_1)$ and $f^k(v_1) \neq f^k(\beta_1)$.
\end{proof}

We remark that we don't actually need the above result, since $P_f$ contains a fixed point, we may take $n$ large enough so that $nk > m$, and then apply the analysis given below to the map $f^{nk}$. 

To prove the case where $k\geq m>2$, we first need to count the number of critical points in the fiber above each element of $P_f = V_{f^k}$.
Using this notation, we have the following.

\begin{lemma}\label{l:numberofcritpointsinfiber}
Let $k > m > 2$. Then:
\begin{enumerate}
 \item there are exactly $2^{k-1}$ critical points in the fiber above $z$ if and only if $z \in \{ v_1,v_2, \beta_1 \}$.
 \item for $2 \leq j \leq m-1$ there are exactly $2^{k-j}$ critical points in the fiber above $\beta_{j}$. 
  \item there are exactly $2^{k-(m-1)} - 2$ critical points above $\alpha = f^m(v_1)$.
\end{enumerate}
\end{lemma}

\begin{proof}\mbox{}
\begin{enumerate}
 \item This first claim follows from Lemma~\ref{l:coalescing3pts}.
 \item We proceed by induction on $k$. For $k=m+1$ it is easy to verify the claim, so assume that for some $k \geq m+1$ the statement holds. Observe that the fiber over $\beta_{j}$ under $f^{k+1}$ is the union of the fibers over the elements of $f^{-1}(\beta_j)$ under $f^k$ and that $f^{-1}(\beta_j) = \{ \beta_{j-1}, \zeta \}$ where $\zeta \notin P_f$. Thus by the observation the fiber over $\beta_j$ under $f^{k+1}$ is equal to the union of the fibers over $\beta_{j-1}$ and $\zeta$ under $f^k$. Since $\zeta$ is not postcritical, there are no critical points in the fiber over it. Therefore the critical points in the fiber over $\beta_{j}$ under $f^{k+1}$ are precisely those over $f^{k}$. The claim follows by the inductive hypothesis.
 \item To get the last claim, we can use the fact that the total number of critical points for $f^k$ is $2^{k+1} - 2$. Summing the number of critical points (which are all simple) from the first two cases, we see there are exactly $2^{k-(m-1)} - 2$ critical points unaccounted for. These must lie in the fiber over $\alpha$.
\end{enumerate}
\end{proof}

We are now able to prove the following.

\begin{proposition}\label{p:detectcritptscritcoalfixed}
If $k > m > 2$, we can detect the critical points and critical values of $f$ from $f^k$. 
\end{proposition}

\begin{proof}
Using Lemma~\ref{l:numberofcritpointsinfiber}, we can pick out the elements $\alpha$ and $\beta_{m-1}$ in $P_f$ by looking at the number of critical points in the fiber above each point of $P_f$. Hence by Corollary~\ref{cor:detectalphabeta}, we can detect $\mathcal{C}_f$ and $\mathcal{V}_f$.
\end{proof}

We are now ready to prove Theorem~\ref{t:DetectQuadraticCritPts}.

\begin{proof}[Proof of Theorem~\ref{t:DetectQuadraticCritPts}]
The claims in part (i) follow from Lemma~\ref{l:powermaps}, Proposition~\ref{p:dihedralimpliescoalescing}  and Corollary~\ref{cor:detectcpsnocritcoal}. Parts (ii) and (iii) follow from the combination of Propositions~\ref{p:characterizecritcoal},~\ref{p:critcoaldecks},~\ref{p:detectcritptscritcoalfixed}, along with Lemmas~\ref{l:detectcpscvsmequals2}, \ref{l:notneeded} and Corollary~\ref{cor:detectcpscritcoalnofixed}. 
\end{proof}

This also completes the proof of Thoerem~\ref{mthm3}.

%%%%%%%%%%%%%%%%%%%%%%%%
\section{Bicritical rational maps with shared iterates}\label{s:sharediterates}

We begin this section with a proof of Theorem~\ref{mthm1}.

\begin{proof}[Proof of Theorem~\ref{mthm1}]
Proposition \ref{p:detectcpscvsnonquadratic} gives the result in the case that $f$ and $g$ have degree $d \geq 3$. 
Now consider the case $d=2$. In the non-critically coalescing case, $f^k$ uniquely determines $\mathcal{V}_f$ by Lemma \ref{l:nontypealphacandetectcvs} and uniquely determines 
$\mathcal{C}_f$ by Corollary~\ref{cor:detectcpsnocritcoal}. In the critically coalescing case, the fact that $f^k$ uniquely determines $\mathcal{C}_f$ and $\mathcal{V}_f$ follows from a combination of Lemma~\ref{l:detectcpscvsmequals2}, Lemma~\ref{l:notneeded} and Proposition~\ref{p:detectcritptscritcoalfixed}. \end{proof}

We remark that the converse to Theorem \ref{mthm1} does not hold. For an example, consider $f(z) = z^2$ and $g(z) = -z^2$. Observe that $\mathcal{C}_f = \mathcal{C}_g = \mathcal{V}_f = \mathcal{V}_g = \{ 0 , \infty \}$. However, for $k \geq 1$ we have $f^k(1) = 1$, but $g^k(1) = -1 \neq 1$ and so $f^k \neq g^k$. An example where $f$ and $g$ are not power maps is given below Theorem~\ref{t:firstcousinsevendegree}. We will use Theorem~\ref{mthm1} to help us prove Theorem~\ref{mthm2}. The following Lemma is classical.

\begin{lemma} \label{l:criticalpointsgiveMobius}
Let $f$ and $g$ be bicritical rational maps such that $\mathcal{C}_f = \mathcal{C}_g$.  Then $g = \mu \circ f$ for some M\"obius transformation $\mu$ sending $\mathcal{V}_f$ to $\mathcal{V}_g$. 
\end{lemma}

\begin{proof} 

Let $c_1$ and $c_2$ be the two (distinct) critical points of $f$ and $g$ and let $a$ be an arbitrary point in $\hat{\mathbb{C}} \setminus \{c_1,c_2\}$. Note that $f(c_1)$, $f(c_2)$ and $f(a)$ are three distinct points in $\hat{\mathbb{C}}$ (because $f$ is $d$-to-$1$, counted with multiplicity). Let $z:\hat{\mathbb{C}} \to \hat{\mathbb{C}}$ be the M\"obius transformation satisfying 
\[z\circ f(c_1) = 0,\quad z\circ f(c_2) = \infty \textrm{ and } z\circ f(a) = 1.\] Similarly, $g(c_1)$, $g(c_2)$ and $g(a)$ are three distinct points in $\hat{\mathbb{C}}$. Let $w:\hat{\mathbb{C}} \to \hat{\mathbb{C}}$ be the M\"obius transformation satisfying 
\[w\circ g(c_1) = 0,\quad w\circ g(c_2) = \infty \textrm{ and } w\circ g(a) = 1.\] 
Then the meromorphic functions $z\circ f$ and $w\circ g$ have  $d$-fold zeroes at $c_1$ and $d$-fold poles at $c_2$. It follows that their quotient is constant. Since they coincide at $a$, they are equal. Set $\mu := w^{-1}\circ z:\hat{\mathbb{C}} \to \hat{\mathbb{C}}$. Then, $g = \mu \circ f$. In addition, 
\[\mu(\mathcal{V}_f) = \mu \circ f(\mathcal{C}_f) = g(\mathcal{C}_f) = g(\mathcal{C}_g) = \mathcal{V}_g.\]
 \end{proof}

We note that the conclusion of Lemma \ref{l:criticalpointsgiveMobius} does not hold for generic rational maps (see e.g. \cite{Goldberg}), though it does hold for all polynomials (\cite{ZakeriCritical}).  

\begin{lemma} \label{l:pakovich1}
Suppose $f$ is surjective and $f^k = (\mu \circ f)^k$ for $k \in \mathbb{N}$ and $\mu$ a M\"obius transformation.  Then 
\begin{enumerate} 
\item $f^k = (f \circ \mu)^k$, and 
\item $f^k \circ \mu = \mu \circ f^k$
\end{enumerate}

\end{lemma}

\begin{proof}
Since $f$ is surjective, we can cancel one $f$ from the right side of
$$f^k =  (\mu \circ f)^k = (\mu \circ f)^{k-1} \circ \mu \circ f$$
to obtain 
\begin{equation} \label{eq:onedegreeless} 
f^{k-1} = (\mu \circ f)^{k-1} \circ \mu.
\end{equation} Therefore, postcomposing both sides with $f$ yields part (i),
 $$f^k =  f \circ (\mu \circ f)^{k-1} \circ \mu = (f \circ \mu)^k.$$
 
 Next, 
 $$\mu^{-1} \circ  f^k \circ \mu  = \mu^{-1} \circ (\mu \circ f)^k \circ \mu = (f \circ \mu)^k =f^k, $$
 with the leftmost equality due to the assumption $f^k = (\mu \circ f)^k$ and the rightmost equality due to part (i). 
\end{proof}

\begin{lemma} \label{l:MobiusInvolution}
Let $f$ and $g$ be bicritical rational maps, neither of which is a power map, such that $f^k = g^k$ for some $k \in \mathbb{N}$.  Then either $f = g$ or there exists a M\"obius involution $\mu$ such that 
\begin{enumerate}
\item  $g = \mu \circ f$, and 
\item $\mu$ fixes both $\mathcal{C}_f$ and $\mathcal{V}_f$ as sets. 
\end{enumerate}
\end{lemma}

\begin{proof}
By Theorem~\ref{mthm1}, $\mathcal{C}_f = \mathcal{C}_g$ and $\mathcal{V}_f = \mathcal{V}_g$.  Therefore Lemma \ref{l:criticalpointsgiveMobius} guarantees that there exists a M\"obius transformation $\mu$ such that $g = \mu \circ f$ and $\mu$ fixes $\mathcal{V}_f = \mathcal{V}_g$ as a set. To show that $\mu$ is an involution, we consider two cases. 
\begin{itemize}
 \item \emph{Case 1:} $\mu$ interchanges the points of $\mathcal{V}_f$.  Pick $z$ to be a fixed point of $\mu$; then $z$ and the two critical values $\mathcal{V}_f$ are all distinct.  But $\mu^2$ fixes each of these three points, so $\mu^2 = \textrm{Id}$. 
\item \emph{Case 2:} $\mu$ fixes $\mathcal{V}_f$ pointwise.  We have $f^k = g^k = (\mu \circ f)^k$ by assumption, so by Lemma  \ref{l:pakovich1} part (i) we have
$f^k = (f \circ \mu)^k$.  Note that $f \circ \mu$ is a bicritical rational map, so  Theorem~\ref{mthm1} implies the leftmost equality of 
$$\mathcal{C}_{f} = \mathcal{C}_{f \circ \mu} = \mu^{-1}(\mathcal{C}_f).$$  Therefore $\mu$ fixes $\mathcal{C}_f$ as a set; either $\mu$ interchanges the points of $\mathcal{C}_f$ or it fixes them pointwise.  Either way, $\mu^2$ fixes $\mathcal{C}_f$ pointwise.  Since $f$ is assumed to not be a power map, at least one point of $\mathcal{C}_f$ is not in $\mathcal{V}_f$. Thus $\mu^2$ fixes pointwise at least three distinct points (a point of $\mathcal{C}_f$ and both points of $\mathcal{V}_f$), so $\mu^2 = \textrm{Id}$. 
\end{itemize}
It remains to prove that $\mu$ fixes $\mathcal{C}_f$ as a set. Since $\mu$ is an involution, $\mu \circ f^k  \circ \mu= (\mu \circ f \circ \mu)^k$.
By Lemma \ref{l:pakovich1},
 \begin{equation} \label{eq:frompak1} \mu \circ f^k  \circ \mu = \mu \circ \mu \circ f^k = f^k. \end{equation}
So $f^k = (\mu \circ f \circ \mu)^k$.  Theorem~\ref{mthm1} gives $\mathcal{C}_f = \mathcal{C}_{\mu \circ f \circ \mu}$ (as well as $ \mathcal{V}_f = \mathcal{V}_{\mu \circ f \circ \mu}$).
This implies that $\mu$ fixes $\mathcal{C}_f$ as a set. 
\end{proof}

\begin{lemma} \label{l:evendegree}
Let $f$ and $g$ be bicritical rational maps such that 
\begin{enumerate}
\item neither $f$ nor $g$ is a power map, 
\item  $f^k = g^k$ for some $k \in \mathbb{N}$,
\item the degree of $f$ and $g$ is even, 
\item $g = \mu \circ f$ for some nonidentity M\"obius transformation $\mu$ that fixes both $\mathcal{C}_f$ and $\mathcal{V}_f$ as sets.  
\end{enumerate}
Then $\mu$ transposes the elements of $\mathcal{V}_f$ and transposes the elements of $\mathcal{C}_f$. 
\end{lemma}

Note the assumption of even degree in Lemma \ref{l:evendegree}.

\begin{proof}
If $\mu$ fixes the elements of $\mathcal{C}_f$ and $\mathcal{V}_f$ pointwise, then $\mu$ is the identity. First suppose that $\mu$ fixes $\mathcal{V}_f$ pointwise, but transposes the elements of $\mathcal{C}_f$. Then $f^k(v_i)$ is fixed under $f^{k}$, and so $f$ is postcritically finite. Since $\mu$ is not the identity, we must have $f^k(v_i) \in \{ v_1, v_2 \}$. We split into cases.
  
  \begin{itemize}
   \item \emph{Case 1}. $f^k(v_1) = v_1$. In this case we must also have $f^k(c_1) = c_1$. But then
   \[
    c_2 = \mu(c_1) = \mu (f^k(c_1)) = f^k(\mu(c_1)) = f^k(c_2).
   \]
Thus since $f^k(c_2) = c_2$ we have $f^k(v_2) = v_2$.
 \item \emph{Case 2.} $f^k(v_1) = v_2$. Then we must have $f^k(c_1) = c_2$, and then a similar computation to the above gives $f^k(c_2) = c_1$.
  \end{itemize}
In either case we have $f^{2k}(c_i) = c_i$. Now note that $f^{2k}$ has $d^{2k} + 1$ fixed points (counting multiplicity). However, if $f^k$ had repeated fixed points, then $f^k$ would have a parabolic fixed point, and so could not be postcritically finite. But this contradicts the fact that $f$ is postcritically finite. Thus $f^{2k}$ has exactly $d^{2k}+1$ fixed points.

To complete the argument, note that since $\mu$ commutes with $f^{2k}$, it must permute the $d^{2k}+1$ fixed points of $f^{2k}$. But since $\mu$ is an involution, all points must have period $1$ or $2$ under $\mu$. By assumption, $v_1$ and $v_2$ are fixed under $\mu$. However, since $d^{2k} + 1 - 2 = d^{2k} - 1$ is odd, there must be another fixed point of $f^{2k}$ which is fixed by $\mu$. But then $\mu$ has three fixed points, and so must be the identity. This is a contradiction.
  
  One can prove the case where $\mu$ is an involution which fixes $\mathcal{C}_f$ pointwise and transposes the elements of $\mathcal{V}_f$ in a similar way to the above. However, a quicker argument is as follows. In this case, we know that $\mu$ must belong to the deck group of $f$, so that $f \circ \mu = f$. But then we have $\mu \circ f^k = f^k \circ \mu = f^k$, which is true if and only if $\mu$ is the identity. Once again we have obtained a contradiction.
\end{proof}

\begin{theorem}\label{t:firstcousinsevendegree}
If $f$ and $g$ are bicritical rational maps of even degree, and neither $f$ nor $g$ is a power map, and $f^k = g^k$ for some $k \in \mathbb{N}$, then 
 $f^2 = g^2$. 
\end{theorem}

\begin{proof}
By Lemma \ref{l:MobiusInvolution}, either $f = g$ or $g = \mu \circ f$ for some M\"obius involution $\mu$ that fixes both $\mathcal{C}_f$ and $\mathcal{V}_f$ as sets.  If $f=g$ we are done, so assume the latter. 
%We want to show $$\mu \circ f \circ \mu \circ f  = f^2.$$ It suffices to show $\mu \circ f \circ \mu = f$. 
%Since $\mu$ is an involution, $\mu \circ f^k  \circ \mu= (\mu \circ f \circ \mu)^k$.
By Lemma \ref{l:pakovich1},
 \begin{equation}\mu \circ f^k  \circ \mu = \mu \circ \mu \circ f^k = f^k. \end{equation}
So $f^k = (\mu \circ f \circ \mu)^k$.  Theorem~\ref{mthm1} gives $\mathcal{C}_f = \mathcal{C}_{\mu \circ f \circ \mu}$ and $ \mathcal{V}_f = \mathcal{V}_{\mu \circ f \circ \mu}$.
%This implies that $\mu$ fixes $\mathcal{C}_f$ as a set, and also fixes $\mathcal{V}_f$ as a set. 

Since $f$ is not a power map, $\mu \circ f \circ \mu$ is also not a power map.  Then  Lemma \ref{l:MobiusInvolution} gives that either $f = \mu \circ f \circ \mu$ or there exists a M\"obius involution $\nu$ such that  
\begin{equation} \label{eq:NuMuMess} f = \nu \circ  \mu \circ f \circ \mu \end{equation}
and $\nu$ fixes $\mathcal{C}_f$ and $\mathcal{V}_f$ as sets.  If $f = \mu \circ f \circ \mu$ we are done (since then $\mu \circ f \circ \mu \circ f  = f^2$), so assume such $\nu$ exists.   

Equation \eqref{eq:NuMuMess}   implies 
$$f = \nu \circ  \mu \circ (\nu \circ  \mu \circ f \circ \mu) \circ \mu =  (\nu \circ  \mu)^2  \circ f $$
Since $f$ is surjective, this implies $ ( \nu \circ  \mu)^2 = \textrm{Id}$, and hence $\nu \circ \mu = \mu \circ \nu$. 

Note that if $x$ is a fixed point of $\nu$,  then $\mu(x)$ is a fixed point of  $\mu \circ \nu  \circ \mu^{-1} = \nu \circ \mu \circ \mu^{-1} = \nu$, i.e. $\mu$ sends fixed points of $\nu$ to fixed points of $\nu$.  Hence $\mu$ fixes setwise the set of fixed points of $\nu$; similarly, $\nu$ fixes setwise the set of fixed points of $\mu$.  So either the fixed points of $\nu$ and $\mu$ coincide, or $\nu$ and $\mu$ interchange each other's fixed points. 

\medskip
\emph{Case 1:} $\mu$ and $\nu$ share the same set of fixed points.  Then, since M\"obius involutions are determined by their two fixed points, $\mu = \nu$.  So \eqref{eq:NuMuMess} gives $f = f \circ \mu$. Then from \eqref{eq:frompak1}
$$\mu \circ f^k = \mu \circ f^k \circ \mu = f^k,$$
so we may cancel a factor of $f^k$ from both sides, obtaining $\mu = \nu = \textrm{Id}$ and $f = g$. 

\medskip
\emph{Case 2:} $\mu$ and $\nu$ interchange each other's fixed points.  By Lemma \ref{l:evendegree}, $\mu$ interchanges the points of $\mathcal{C}_f$ and interchanges the points of $\mathcal{V}_f$.
Then $$v_i = f(c_i) = \nu \circ \mu \circ f \circ \mu (c_i) = \nu \circ \mu \circ f(c_{\neq i}) =  \nu \circ \mu (v_{\neq i}) = \nu(v_i),$$ meaning $\nu$ fixes $\mathcal{V}_f$ pointwise.  Without loss of generality (by conjugating $f$), assume $\mathcal{C}_f = \{0,\infty\}$ and $v_1=1$. 
 The assumption that $\mu$ interchanges $0$ and $\infty$ implies that $\mu$ has the form $\mu(z) = k/z$ for some $k \in \mathbb{C}$.  Then other critical value is $\mu(1) = k (\neq 1)$.  Now $\nu$ is an involution that fixes $1$ and $k$ pointwise, and fixes $\{0,\infty\}$ as a set.
Since $\nu$ fixes $0$ and $\infty$ as points, then $\nu$ has the form $\nu(z) = z k_2$ for some $k_2 \in \mathbb{C}$; but then the assumption that $\nu$ fixes $1$ as a point implies $k_2 = 1$, i.e. $\nu = \textrm{Id}$. Thus $f =  \mu \circ f \circ \mu$, as desired. 
\end{proof}

\color{black}

We remark that the conclusion of Theorem~\ref{t:firstcousinsevendegree} is not true in the odd degree case.

\begin{example} \label{ex:oddDegree4thIterate}
 Let $f(z) = \frac{z^3 - 1}{z^3 + 1}$ and $g(z) = - f(z)$. It is easy to see that $\mathcal{C}_f = \mathcal{C}_g = \{0, \infty\}$ and $\mathcal{V}_f = \mathcal{V}_g = \{ -1,1\}$. The critical portrait for $f$ is 
 \[
\xymatrix{
0 \ar[r]^{2:1}  & 1  \ar[r]   & \infty \ar[r]^{2:1}  & -1  \ar@/^1pc/[lll]
\\
}
\]
and the critical portrait for $g$ is 
 \[
\xymatrix{
0 \ar@/^/[r]^{2:1}  & -1  \ar@/^/[l]   & \infty \ar@/^/[r]^{2:1}  & 1  \ar@/^/[l]
}.
\]
 Since $f(f(0)) \neq g(g(0))$, we see that $f^2 \neq g^2$. On the other hand, a direct computation shows that $f^4 = g^4$.
\end{example}

As promised, we also include an example to show that the converse of Theorem~\ref{mthm1} does not hold, even if we exclude counterexamples which are power maps.

\begin{example} \label{ex:noSharedIterate}
Here we provide an example of bicritical rational maps $f$ and $g$ such that $\mathcal{C}_f = \mathcal{C}_g$ and $\mathcal{V}_f = \mathcal{V}_g$ but $f$ and $g$ do not share an iterate. Let $f(z) = \frac{2(z^2 - 1)}{16z^2-1}$ and $g(z) = \frac{z^2-16}{8(z^2-1)}$. Then we have $\mathcal{C}_f = \mathcal{C}_g =  \{ 0 , \infty \}$ and $\mathcal{V}_f = \mathcal{V}_g = \left\{ \frac{1}{8}, -2 \right\}$. However, a quick computation yields
\[
 f^2(z) = \frac{2(84z^4 - 8z^2-1)}{64z^4+32z^2-21}
\]
whereas
\[
g^2(z) = \frac{341z^4-672z^2+256}{8(21z^4-32z^2-64)}.
\]
Since $f^2 \neq g^2$ it follows from Theorem~\ref{t:firstcousinsevendegree} that $f^k \neq g^k$ for all $k \geq 1$.
\end{example}

Our current results allow us to complete the proof of Theorem~\ref{mthm2}.

\begin{proof}[Proof of Theorem~\ref{mthm2}]
 The first claim is precisely that of Theorem~\ref{t:firstcousinsevendegree}. We can henceforth assume that $f^2 =g^2$. To prove the second claim, note that by Lemma~\ref{l:MobiusInvolution}, we must have $g = \mu \circ f$ for some involution $\mu$. But then
 \[
  f^2 = g^2 = \mu \circ f \circ \mu \circ f.
 \]
 Since $f$ is surjective, we may cancel a copy of $f$ on the right to get $f = \mu \circ f \circ \mu$. Since $\mu$ is an involution, we see that $\mu$ is an automorphism of $f$. The case for $g$ is similar.
\end{proof}

\appendix

 \section{Symmetry Locus and Mixing}\label{symlocus}
 
 A motivation for the present work is to lay the foundations for an investigation of the structure of the symmetry locus $\Sigma_d$ in terms of mixings and matings of polynomials.   
 Recall our provisional definition (Definition \ref{def:mixedmating}) that a degree $d$ rational map $F$ is a \emph{mixing}\footnote{Another name for this construction could be the anti-mating. However, we avoid this terminology to avoid confusion with the work of Jung \cite{JungAntiMatings}.}  of postcritically finite degree $d$ polynomials $f$ and $g$ if $F^2 = (f \mate g)^2$ and $F \neq f \mate g$ for some for geometric mating $f \mate g$ of $f$ and $g$. (See \cite{MilnorMating,ShishikuraTanMating} for definitions and background on matings). This section contains mainly conjectures and observations obtained from looking at computer pictures. We hope to give a more rigorous treatment of these ideas in a later work.
 
%Recall that the symmetry locus in the space of degree $d$ bicritical rational maps consists of all such maps which commute with a non-identity M\"obius transformation. According to Milnor \cite{MilnorBicritical}, the symmetry locus in the space of degree $d$ bicritical rational maps is an irreducible variety when $d$ is even, but splits into two irreducible components $\Sigma_d^+$ and $\Sigma_d^-$ when $d$ is odd. 
%\subsection{Mixing Polynomials}

%The notion of a mixing\footnote{Another name for this construction could be the anti-mating. However, we avoid this terminology to avoid confusion with the work of Jung \cite{JungAntimating}.} of two polynomials is closely related to that of matings. For the moment, we give a provisional definition for what it means for a rational map to be a mixing. We denote by $f \mate g$ the mating of two polynomials (see \cite{MilnorMating,ShishikuraTanMating} for definitions and background on matings).

%\begin{definition}
%Let $F$ be a rational map of degree $d$ and suppose there exist postcritically finite degree $d$ polynomials $f$ and $g$ such that
%\begin{enumerate}
%\item $F^2 = (f \mate g)^2$
%\item $F \neq f \mate g$.
%\end{enumerate}
%Then we say $F$ is the mixing of $f$ and $g$ and write $F = f \otimes g$.
%\end{definition}

The notion of the mixing of two polynomials seems to be very rich. For simplicity, we restrict the present discussion to the degree $2$ case. Recall that the symmetry locus in degree 2, $\Sigma_2$, may be parameterised by $c$ via the map $f_c(z) = c(z + 1/z)$. Such a map has critical points at $-1$ and $1$. It is not hard to see that there are many matings in the space $\Sigma_2$. Indeed, it can be shown that if $f$ is a postcritically finite quadratic polynomial, then if $f \mate f$ is not obstructed (equivalently, $f$ does not belong to the $1/2$-limb of the Mandelbrot set) then the mating $F = f \mate f$ belongs to $\Sigma_2$. However, there exist matings in $\Sigma_2$ which are not self-matings, as we show below.

We give a number of examples of mixings and their corresponding matings. Claims in these examples are given without proof, but may be verified by the assiduous reader. We include images showing the Julia sets, with arrows indicating the critical orbits of the maps.

\begin{example}
When $c \approx 0.221274 + 0.48342i$, the map $f_c$ is the self-mating of Douady's rabbit. Since $f_c$ is a mating and $f_c$ is a hyperbolic map, the forward orbits of the critical points $-1$ and $1$ are disjoint. 
\begin{figure}[ht!]
\centering
\begin{subfigure}{0.45\textwidth}
\includegraphics[width=0.9\linewidth]{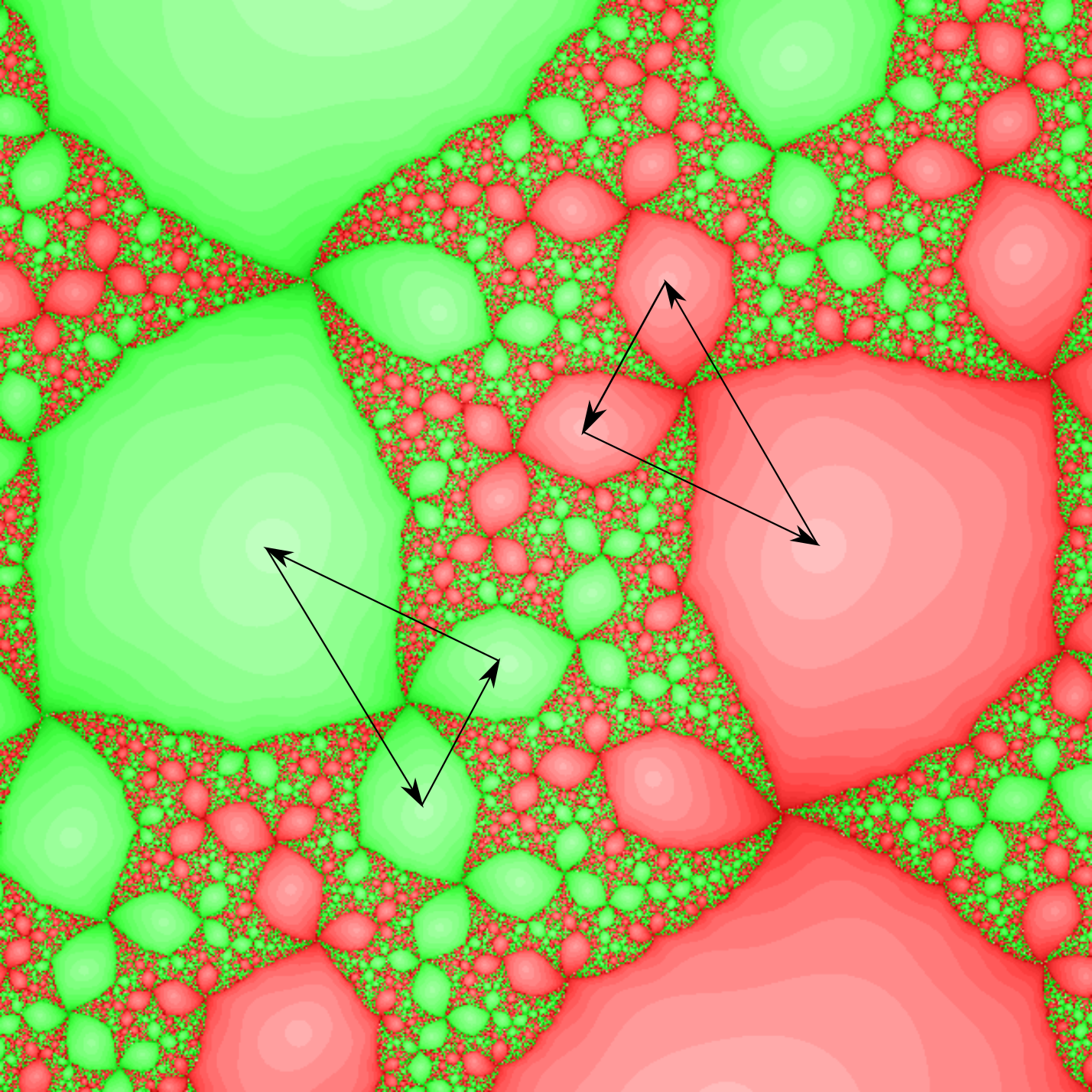}
\end{subfigure}
~
\begin{subfigure}{0.45\textwidth}
\includegraphics[width=0.9\linewidth]{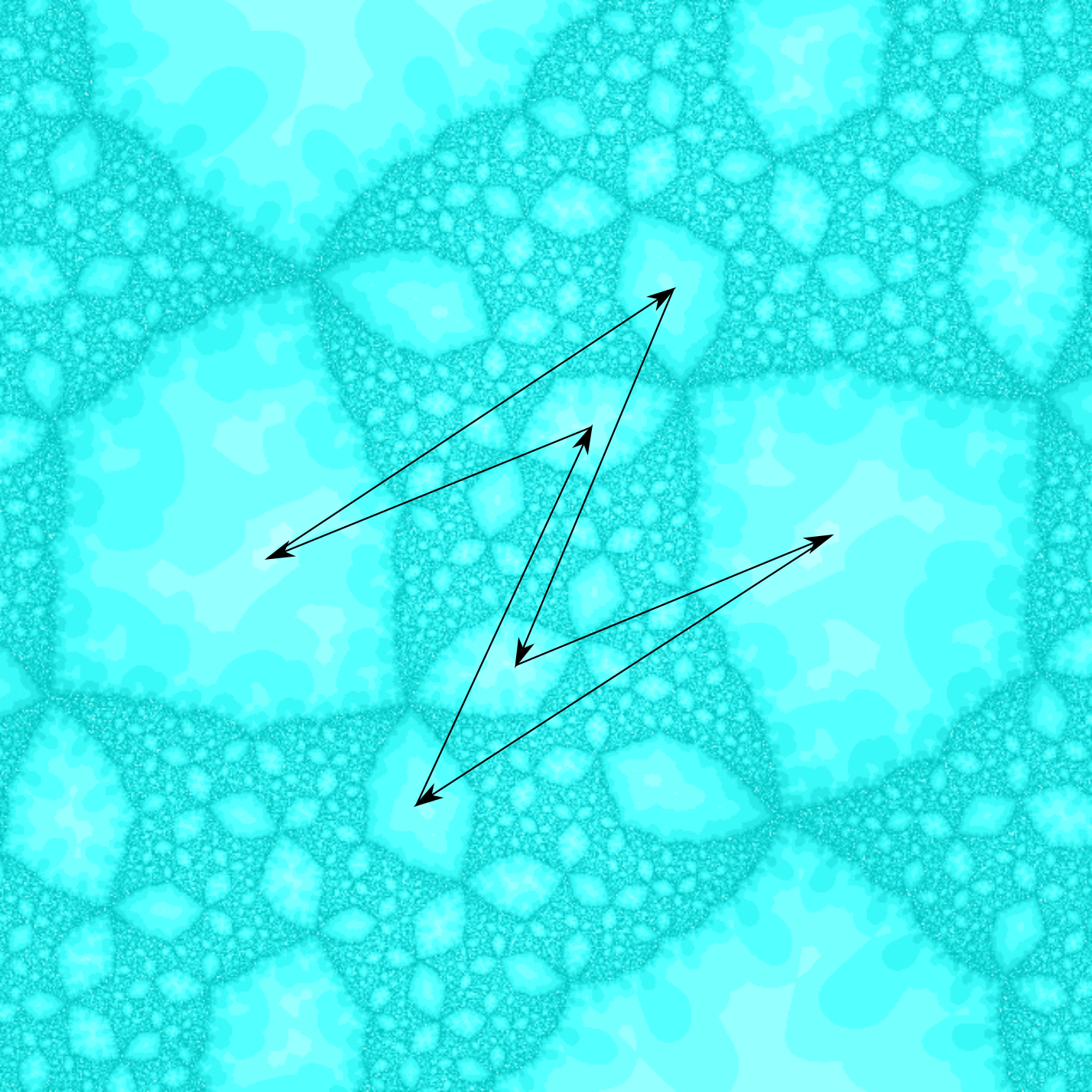}
\end{subfigure}
\caption{The Julia sets for the self-mating and self-mixing of Douady's rabbit.} 
\label{f:rabrab}
\end{figure} 
Indeed, both critical points belong to a period $3$ superattracting cycle. The map $f_{-c}$ is also a hyperbolic map, but it is not a mating since $f^3(-1) = 1$ and $f^3(1)=-1$, so the two critical points belong to the same period $6$ superattracting cycle. Accordingly, we say that $f_{-c}$ is the self-mixing of Douady's rabbit; see Figure~\ref{f:rabrab}.
\end{example}

\begin{example}
There exist matings in $\Sigma_2$ which are not self-matings. For a particular example, take $c \approx -0.471274 - 0.813859i $. This is the mating of Douady's rabbit with the airplane polynomial (or, equivalently, the mating of the airplane polynomial with Douady's rabbit, since these maps are equal by the results of \cite{Clusters}).
\begin{figure}[ht!]
\centering
\begin{subfigure}{0.45\textwidth}
\includegraphics[width=0.9\linewidth]{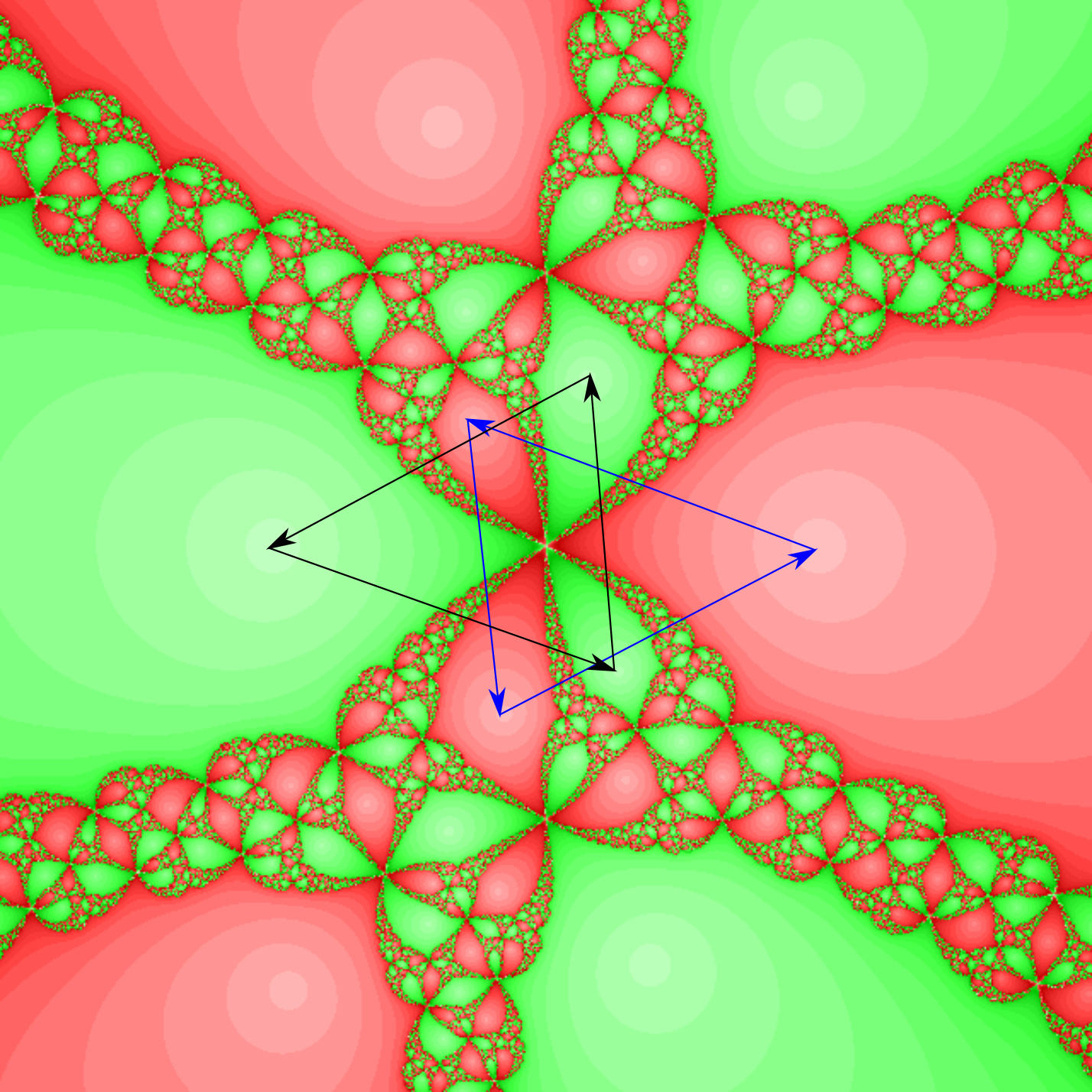}
\end{subfigure}
~
\begin{subfigure}{0.45\textwidth}
\includegraphics[width=0.9\linewidth]{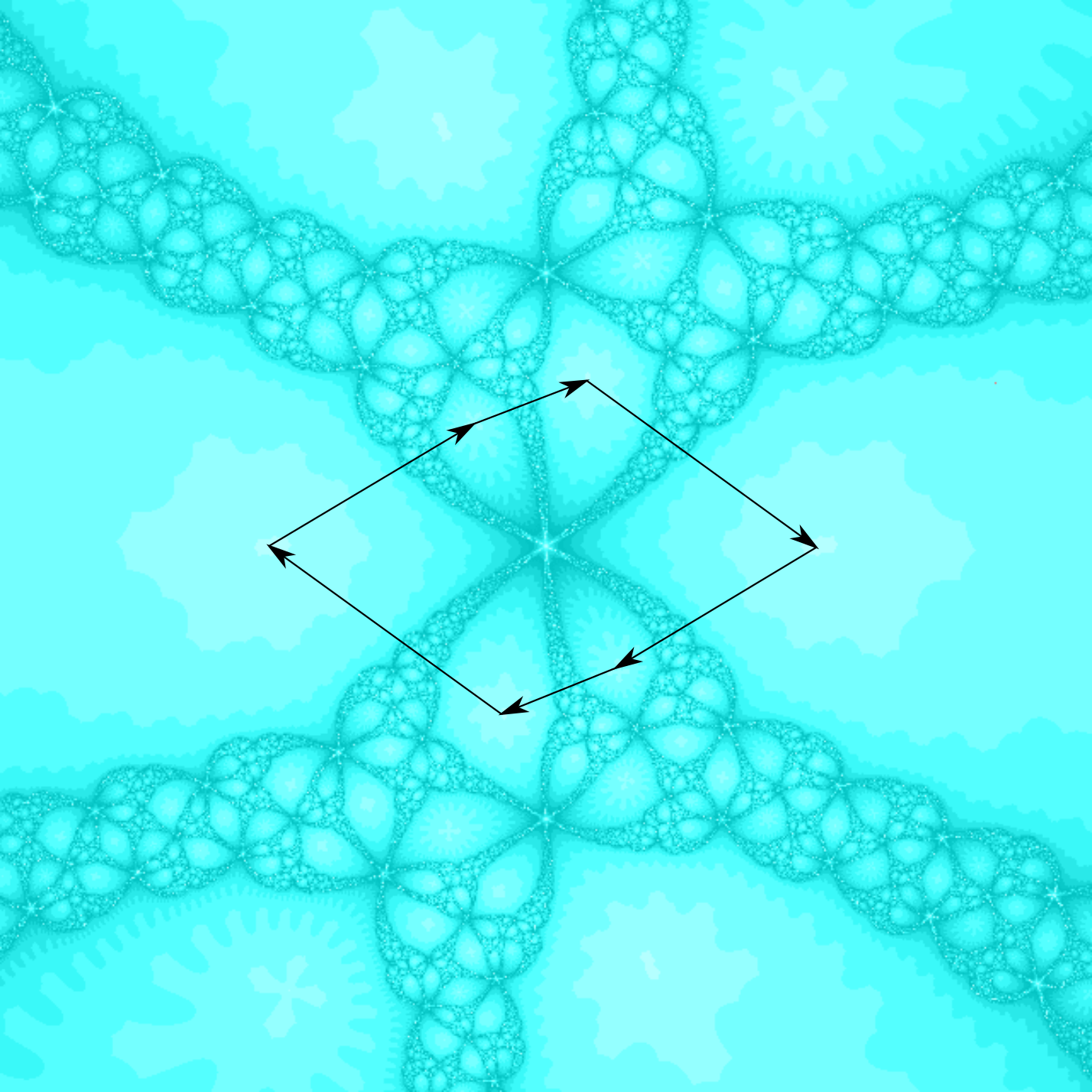}
\end{subfigure}
\caption{The Julia sets for the mating and mixing of Douady's rabbit with the airplane.} 
\label{f:rabair}
\end{figure} 
As with the previous example, the two critical points belong to disjoint period $3$ superattracting orbits. However, for the map $f_{-c}$, the two critical points belong to the same period $6$ superattracting cycle. Thus $f_{-c}$ is the mixing of Douady's rabbit and the airplane, see Figure~\ref{f:rabair}.
\end{example}

%\begin{example}
%Even though hyperbolic maps which have matings need to have disjoint critical orbits, it is not the case (as in the above two examples) that the critical points in a hyperbolic mixing have to belong to the same orbit. To see this, consider $f_c$ for $c= $. In this case, both $f_c$ and $f_{-c}$ have a pair of critical orbits of period $4$. However, as can be seen from the diagram, the dynamics is different for each map.
%\end{example}

\begin{example}
It is possible to be a mixing and a mating. Let $c \approx 0.661848i$. Then $f_c$ is the self-mating of Kokopelli. On the other hand, $f_{-c}$ is the self-mating of co-Kokopelli. Accordingly, we see that $f_c$ is the self-mixing of co-Kokopelli and $f_{-c}$ is the self-mixing of Kokopelli, see Figure~\ref{f:kokoself}.
\begin{figure}[ht!]
\centering
\begin{subfigure}{0.45\textwidth}
\includegraphics[width=0.9\linewidth]{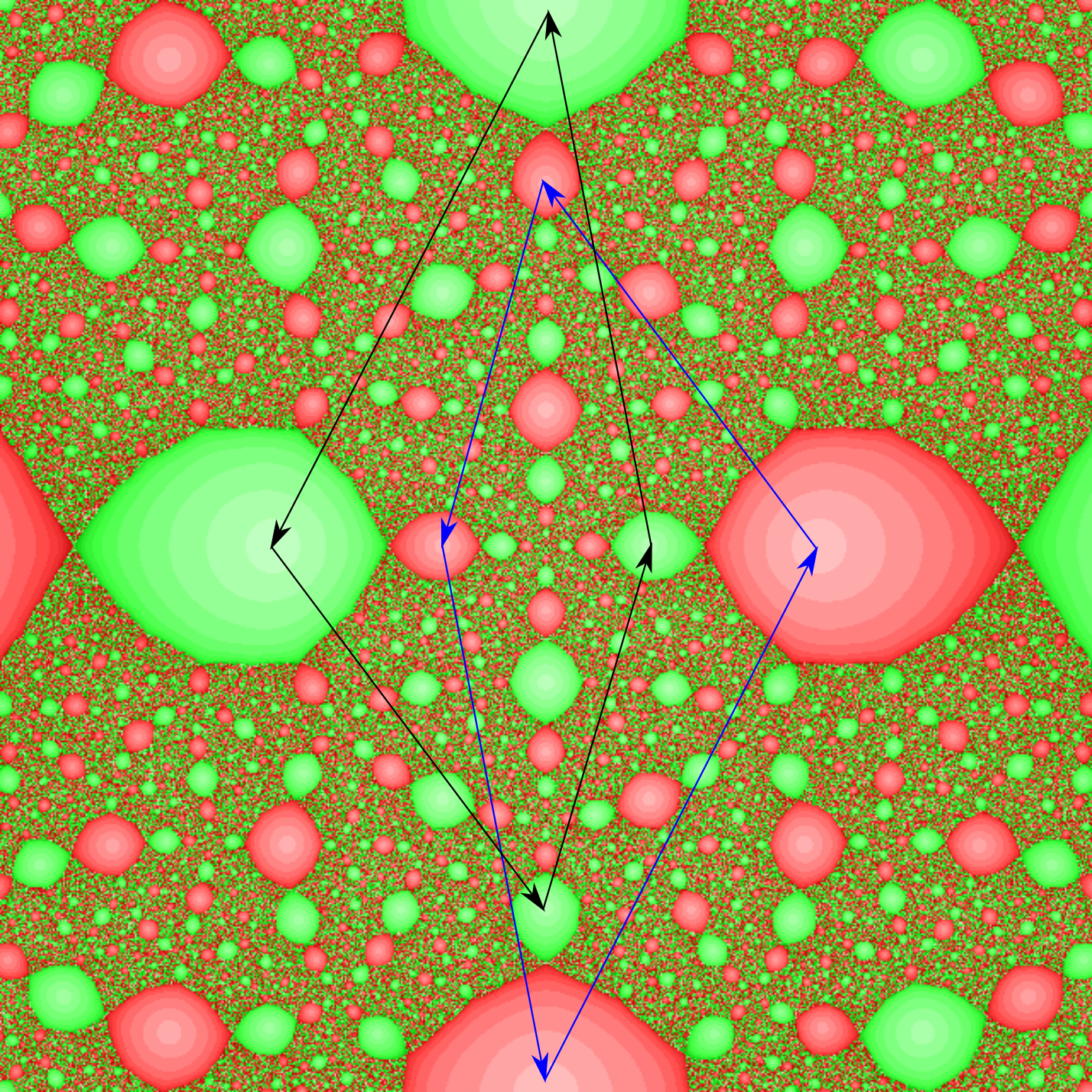}
\end{subfigure}
~
\begin{subfigure}{0.45\textwidth}
\includegraphics[width=0.9\linewidth]{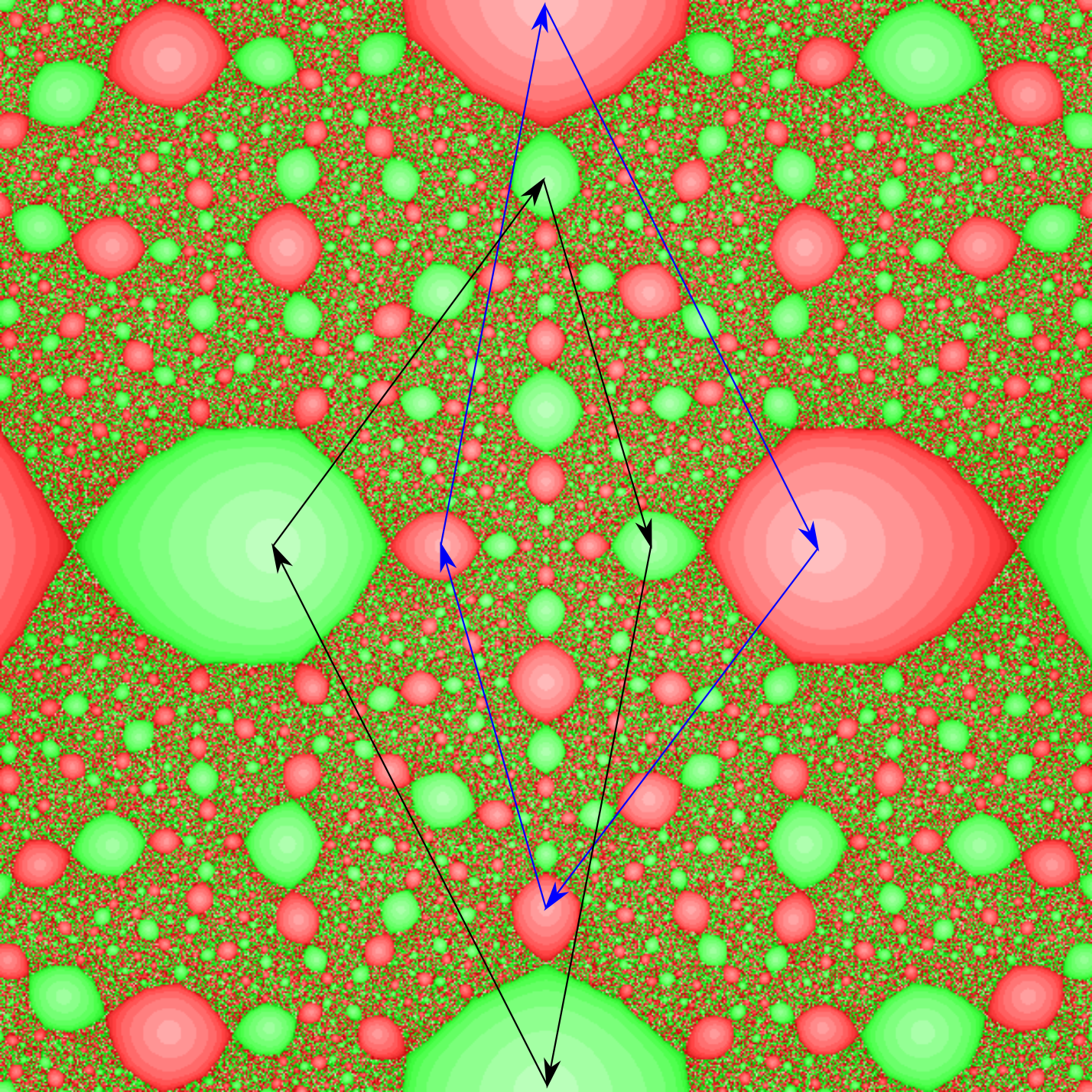}
\end{subfigure}
\caption{The self-mating and self-mixing of Kokopelli.} 
\label{f:kokoself}
\end{figure} 
This example also shows that the critical orbits in a mixing may be disjoint.
\end{example}

\begin{example}
A rather neat example is the following. Let  $c \approx 0.501604 + 0.531587i$. Then $f_c$ is the self-mating of the $1/4$-rabbit. Thus $f_{-c}$ is the self-mixing of the $1/4$-rabbit. However, it was shown by Rees that the map $f_{-c}$ is a shared mating\footnote{An excellent video exhibiting this shared mating can be found on Ch\'eritat's website: \url{https://www.math.univ-toulouse.fr/~cheritat/MatMovies/ReesSharedExample/}}: it is the mating of the double basilica with Kokopelli and the mating of co-Kokopelli with the Airbus polynomial, see Figure~\ref{f:shared}. 
\begin{figure}[ht!]
\centering
\begin{subfigure}{0.45\textwidth}
\includegraphics[width=0.9\linewidth]{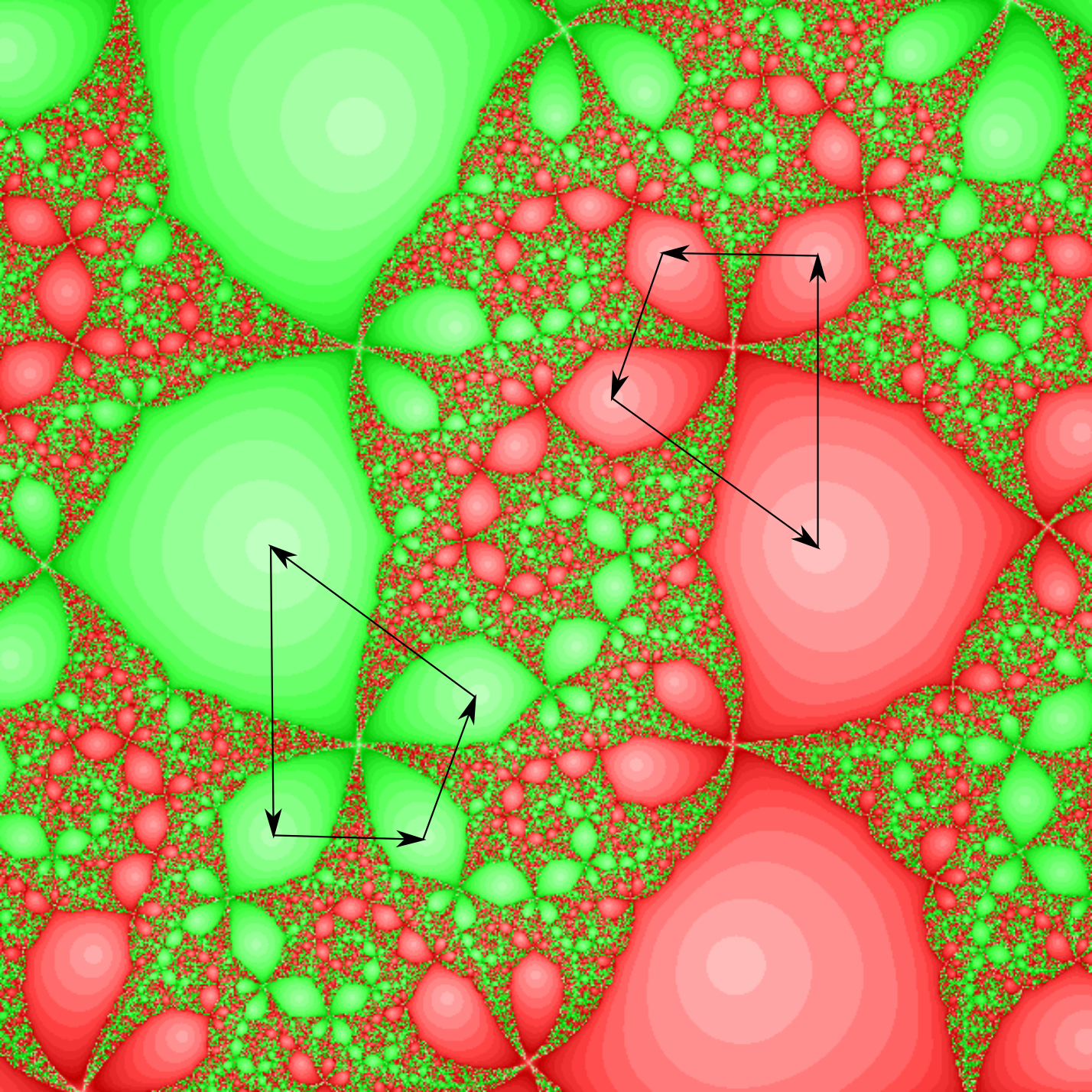}
\end{subfigure}
~
\begin{subfigure}{0.45\textwidth}
\includegraphics[width=0.9\linewidth]{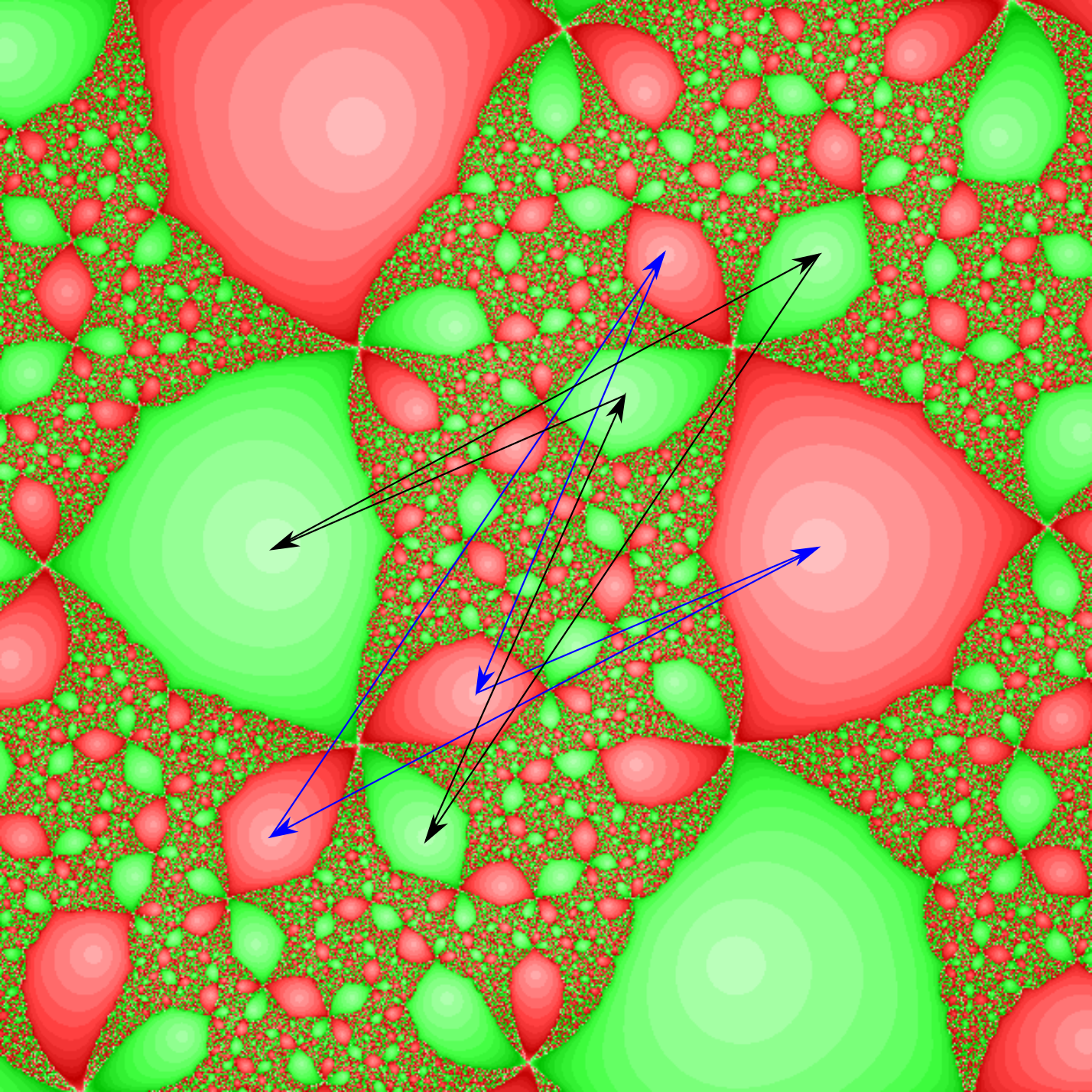}
\end{subfigure}
\caption{The self-mating of the $1/4$-rabbit is a shared mixing!.} 
\label{f:shared}
\end{figure} 
We then may state that $f_c$ is a \emph{shared mixing}, being a mixing of the double basilica with Kokopelli and the mixing of co-Kokopelli with the Airbus polynomial.
\end{example}

We end with a number of questions about mixings, which we hope will be the subject of future work.

\begin{question}
Is there a way of constructing a mixing in an analogous way to the topological mating of the formal mating of two polynomials? If so, for which pairs of polynomials is this construction well-defined? What are the obstructions?
\end{question}

\begin{question}
In \cite{MatingQuestions}, Meyer observed that when $F$ is a degree $d$ rational map with $J(F) = \hatC$, it was sometimes possible to find an \emph{anti-equator}; a simple closed curve which maps (isotopically) onto itself as a $d$-fold cover in an orientation-reversing way. He asked if it were possible to characterize such ``matings''. Could these matings observed by Meyer in fact be mixings? 
\end{question}

\bibliography{Cousins}
\bibliographystyle{plain} 

\end{document}